\documentclass[letterpaper, 10pt, journal, twocolumn, final]{IEEEtran}

\usepackage[T1]{fontenc} 						%
\usepackage[utf8]{inputenc} 				%
\usepackage[english]{babel}
\usepackage{cancel}
\usepackage{amsfonts,amssymb,amsmath,amsthm,mathtools}
\usepackage[hidelinks]{hyperref}
\usepackage{graphicx}
\usepackage{soul}
\usepackage{algorithm,tabularx}
\usepackage[noend]{algpseudocode}
\usepackage[dvipsnames]{xcolor}
\usepackage{tikz,pgfplots}
\usepackage{cite}
\usepackage{lipsum,booktabs}
\usepackage{pifont}
\usepackage{xcolor}

\newcommand{\cmark}{\textcolor{green!80!black}{\ding{51}}}
\newcommand{\xmark}{\textcolor{red}{\ding{55}}}

\newcommand{\n}{n}
\newcommand{\degree}{d}
\newcommand{\dii}{\degree_i}

\newcommand{\dd}{\mathbf{d}}

\newcommand{\cE}{\mathcal{E}}
\newcommand{\cG}{\mathcal{G}}
\newcommand{\cW}{\mathcal{W}}

\newcommand{\cS}{\mathcal{S}}
\newcommand{\cD}{\mathcal{D}}
\newcommand{\cN}{\mathcal{N}}

\newcommand{\cQ}{\mathcal{Q}}
\newcommand{\cH}{\mathcal{H}}

\newcommand{\R}{\mathbb{R}}
\newcommand{\N}{\mathbb{N}}

\newcommand{\x}{x}
\newcommand{\s}{s}

\newcommand{\z}{z}
\newcommand{\y}{y}

\newcommand{\vv}{v}

\newcommand{\pr}{\gamma}

\newcommand{\Ax}{A_{\x}}

\newcommand{\An}{A_{\nabla}}

\newcommand{\iter}{{t}}
\newcommand{\iterp}{{\iter + 1}}

\newcommand{\xt}{\x^{\iter}}

\newcommand{\zt}{\z^{\iter}}

\newcommand{\xtp}{\x^{\iterp}}

\newcommand{\ztp}{\z^{\iterp}}

\newcommand{\xstar}{\x^\star}
\newcommand{\sstar}{\s_\star}
\newcommand{\ystar}{\y_\star}
\newcommand{\sstart}{\sstar^\iter}
\newcommand{\ystart}{\ystar^\iter}

\newcommand{\xjtp}{\x_{j}^{\iter+1}}
\newcommand{\xitp}{\x_{i}^{\iter+1}}
\newcommand{\sitp}{\s_{i}^{\iter+1}}

\newcommand{\zijtp}{\z_{ij}^{\iter+1}}

\newcommand{\yitp}{\y_{i}^{\iter+1}}

\newcommand{\xit}{\x_{i}^{\iter}}
\newcommand{\sit}{\s_{i}^{\iter}}

\newcommand{\zit}{\z_{i}^{\iter}}
\newcommand{\zijt}{\z_{ij}^{\iter}}
\newcommand{\zjit}{\z_{ji}^{\iter}}
\newcommand{\yit}{\y_{i}^{\iter}}

\newcommand{\xjt}{\x_{j}^{\iter}}
\newcommand{\sjt}{\s_{j}^{\iter}}

\newcommand{\yjt}{\y_{j}^{\iter}}

\newcommand{\g}{\ell}

\newcommand{\git}{\g_i^\iter}

\newcommand{\tx}{\tilde{\x}}
\newcommand{\tz}{\tilde{\z}}

\newcommand{\txtp}{\tx^{\iter+1}}
\newcommand{\tztp}{\tz^{\iter+1}}

\newcommand{\txt}{\tx^{\iter}}
\newcommand{\tzt}{\tz^{\iter}}

\newcommand{\bx}{\bar{\x}}
\newcommand{\bz}{\bar{\z}}

\newcommand{\px}{\x_\perp}
\newcommand{\pz}{\z_\perp}

\newcommand{\bzt}{\bz^{\iter}}
\newcommand{\bztp}{\bz^{\iter+1}}

\newcommand{\pxt}{\px^{\iter}}
\newcommand{\pxtp}{\px^{\iter+1}}
\newcommand{\pzt}{\pz^{\iter}}
\newcommand{\pztp}{\pz^{\iter+1}}

\newcommand{\mut}{\bx^\iter}
\newcommand{\mutp}{\bx^{\iterp}}

\newcommand{\tpz}{\tilde{\z}_\perp}

\newcommand{\tpzt}{\tpz^\iter}
\newcommand{\tpztp}{\tpz^{\iterp}}

\newcommand{\Qtpz}{Q_{\tpz}}

\newcommand{\Vpx}{W_{\px}}

\newcommand{\Vmu}{W_{\bx}}

\newcommand{\str}{\mu}

\newcommand{\bb}{b}
\newcommand{\B}{B}
\newcommand{\M}{M}
\newcommand{\rr}{R}

\newcommand{\G}{G}

\newcommand{\one}{\mathbf{1}}

\newcommand{\oneNn}{\one_{N,\n}}

\newcommand{\zeq}{\z^{\text{eq}}}

\newcommand{\pzeq}{\pz^{\text{eq}}}

\newcommand{\norm}[1]{\left \|#1 \right \|}

\newcommand{\T}{^\top}
\newcommand{\inv}{^{-1}}
\newcommand{\map}[3]{#1: #2 \rightarrow #3}
\newcommand{\blkdiag}{\text{blkdiag}}
\newcommand{\until}[1]{\{1,\ldots,#1\}}
\newcommand{\col}{\textsc{col}}

\newcommand{\lit}{\lambda_i(\iter)}
\newcommand{\ljt}{\lambda_j(\iter)}

\newcommand{\lita}{\lambda_{i}(\initer)}

\newcommand{\psiijt}{\psi_{ij}(\iter)}
\newcommand{\bijt}{\beta_{ij}(\iter)}

\newcommand{\Lt}{\Lambda(\iter)}
\newcommand{\Psiit}{\Psi_i(\iter)}
\newcommand{\Psit}{\Psi(\iter)}

\newcommand{\lip}{L}

\newcommand{\F}{F}

\newcommand{\initer}{\tau}
\newcommand{\av}{_{\text{av}}}

\newcommand{\Wa}{W\av}

\newcommand{\bstepa}{\bar{\pr}\av}

\newcommand{\Wai}{W_{i,\text{av}}}

\newcommand{\Psitau}{\Psi(\initer)}

\newcommand{\EL}{\Lambda\av}

\newcommand{\Eli}{\lambda_{i,\text{av}}}

\newcommand{\txa}{\tx\av}

\newcommand{\mua}{\bx\av}
\newcommand{\muat}{\mua^\iter}

\newcommand{\pxa}{\x_{\perp,\text{av}}}

\newcommand{\ra}{r\av}
\newcommand{\rai}{r_{i,\text{av}}}

\newcommand{\txai}{\tx_{i,\text{av}}}
\newcommand{\txaj}{\tx_{j,\text{av}}}

\newcommand{\el}{\epsilon_{\lambda}}

\newcommand\oprocendsymbol{\hbox{$\square$}}
\newcommand\oprocend{\relax\ifmmode\else\unskip\hfill\fi\oprocendsymbol}
\def\eqoprocend{\tag*{$\square$}}
\def\er/{Erd\H{o}s-R\'enyi}

\newcommand{\m}{n_d}

\newcommand{\err}{e}

\newcommand{\errxi}{\err_{i,\x}}
\newcommand{\errzi}{\err_{i,\z}}

\newcommand{\errxit}{\errxi^\iter}
\newcommand{\errzit}{\errzi^\iter}

\newcommand{\errx}{\err_{\x}}
\newcommand{\errz}{\err_{\z}}

\newcommand{\errxt}{\errx^\iter}
\newcommand{\errzt}{\errz^\iter}

\newcommand{\errt}{\err^\iter}
\newcommand{\trt}{\tilde{\err}^\iter}

\newcommand{\fxi}{F_{\xi}}
\newcommand{\Vxi}{V_{\xi}}
\newcommand{\Lxi}{\lip_{\xi}}

\newcommand{\step}{\delta}
\newcommand{\bpr}{\bar{\step}}

\newcommand{\msg}{m}

\newcommand{\msgijt}{\msg_{ij}^\iter}
\newcommand{\msgjit}{\msg_{ji}^\iter}

\newcommand{\hF}{\hat{\F}}
\newcommand{\hFt}{\hF(\iter)}
\newcommand{\hFta}{\hF(\initer)}

\newcommand{\hT}{\hat{T}}
\newcommand{\hTt}{\hT(\iter)}
\newcommand{\hTta}{\hT(\initer)}

\newcommand{\Tmax}{T_{\text{max}}}

\newcommand{\diagentry}[1]{\mathmakebox[1.8em]{#1}}
\newcommand{\xddots}{%
	\raise 4pt \hbox {.}
	\mkern 6mu
	\raise 1pt \hbox {.}
	\mkern 6mu
	\raise -2pt \hbox {.}
}

\graphicspath{{figs/}}

\DeclareMathOperator*{\argmin}{arg\,min}

\def\er/{Erd\H{o}s-R\'enyi}

\newcommand{\GC}[1]{{\color{black} #1}}

\def\algo/{{ADMM-Tracking Gradient}}
\def\algob/{{ADMM-Tracking Gradient V2}}
\def\ralgo/{{Robust ADMM-Tracking Gradient}}
\def\ATG/{{ATG}}
\def\RATG/{{RATG}}
\def\GT/{{Gradient Tracking}}

\allowdisplaybreaks

\newtheorem{theorem}{Theorem}[section]

 \newtheorem{lemma}[theorem]{Lemma}

\newtheorem{assumption}[theorem]{Assumption}
\newtheorem{remark}[theorem]{Remark} 

\title{\algo/\\ for Distributed Optimization over\\ Asynchronous and Unreliable Networks}

\author{Guido Carnevale, Nicola Bastianello, Giuseppe Notarstefano,  Ruggero Carli 
	\thanks{Funded by the European Union - NextGenerationEU under the National Recovery and Resilience Plan (PNRR) - Mission 4 Education and research - Component 2 From research to business - Investment 1.1 Notice Prin 2022 -  DD N. 104 del 2/2/2022, from title ECODREAM Energy COmmunity management: DistRibutEd AlgorithMs and toolboxes for efficient and sustainable operations, proposal code 202228CTKY002 - CUP J53D23000560006.\\
 The work of N. Bastianello was partially supported by the European Union’s Horizon Research and Innovation Actions programme under grant agreement No. 101070162.}
\thanks{Guido Carnevale and Giuseppe Notarstefano are with the Department of Electrical,  Electronic and Information Engineering,  Alma Mater Studiorum - Universita` di Bologna,  Bologna, Italy. Emails: {\tt\footnotesize{name.surname@unibo.it}%
} 
}
	\thanks{Nicola Bastianello is with the School of Electrical Engineering and Computer Science, and Digital Futures, KTH Royal Institute of Technology, Stockholm, Sweden. Email: {\tt\footnotesize nicolba@kth.se}}%
\thanks{Ruggero Carli is with the Department of Information Engineering of the University of Padova, Via G. Gradenigo 6/B, 35131 Padova, Italy. Email: {\tt\footnotesize{carlirug@dei.unipd.it}}%
}
}

\begin{document}

\maketitle
\begin{abstract}
  In this paper, we propose a novel distributed algorithm for consensus
 optimization over networks and a robust extension tailored to deal with asynchronous agents and packet losses.
 \GC{
 Indeed, to robustly achieve dynamic consensus on the solution estimates and the global descent direction, we embed in our algorithms a distributed implementation of the Alternating Direction Method of Multipliers (ADMM).
 }
 Such a mechanism is suitably interlaced with a local proportional
 action steering each agent estimate to the solution of the original
 consensus optimization problem.
 First, in the case of ideal networks, by using tools from system theory, we prove the linear convergence of the scheme with strongly convex costs.
 Then, by exploiting the averaging theory, we extend such a first result to prove that the robust extension of our method preserves linear convergence in the case of asynchronous agents and packet losses.
 Further, by using the notion of Input-to-State Stability, we also guarantee the robustness of the schemes with respect to additional, generic errors affecting the agents' updates.
\GC{Finally, some numerical simulations confirm our theoretical findings and compare our algorithms with other distributed schemes in terms of speed and robustness.}
\end{abstract}

\section{Introduction}

Multi-agent systems have become ubiquitous in many engineering domains, ranging from robotics to power grids, from traffic networks to the Internet of Things. 
Recent technological advances have indeed led to the widespread adoption of devices equipped with computational and communication resources. 
These devices (or agents) then form interconnected systems capable of cooperatively performing a variety of tasks, e.g. learning, resource allocation, and exploration.
Many of these tasks can be formulated as the \emph{consensus (or cost-coupled) optimization} problem
\begin{align*}%
    \min_{\x \in \R^\n} \sum_{i=1}^N f_i(\x),
\end{align*}
where $f_i$ are the local costs of the $N$ agents. This motivated a wide interest in algorithms to enable multi-agent systems to collaborate towards the solution of consensus problems~\cite{notarstefano2019distributed,nedic2018distributed,yang2019survey}.

Different approaches to designing distributed algorithms for consensus optimization have been explored in the literature, with the main ones being decentralized (sub)gradient, Gradient Tracking (GT), and Alternating Direction Method of Multipliers (ADMM)~\cite{notarstefano2019distributed}.
The class of decentralized (sub)gradient methods has the drawback that exact convergence is only achieved using diminishing stepsizes \cite{nedic2018distributed}. Gradient Tracking algorithms were then introduced to guarantee exact convergence with fixed step-sizes, see \cite{shi2015extra,nedic2017achieving,xi2017add}, and have proved their versatility in different setups, see, e.g.,~\cite{bof_multiagent_2019,carnevale2023triggered,dilorenzo2016next,daneshmand2020second,carnevale2022nonconvex,carnevale2022gtadam,yuan_can_2020}.
Besides these gradient-based methods, ADMM, popularized by \cite{boyd_distributed_2010}, has received wide attention and has proven especially suited to distributed setups \cite{mota_d-admm_2013,peng_arock_2016,makhdoumi_convergence_2017,bastianello2020asynchronous}.

However, the successful deployment of distributed algorithms designed according to these approaches is hindered by the practical challenges that multi-agent systems raise.
First of all, the agents cooperating may have highly heterogeneous computational resources, implying that they may perform local updates at different rates. One option is to enforce synchronization of the agents so that they update in lockstep. But besides the technical challenge of synchronization, it would result in faster agents sitting idle before they can initiate a new update \cite{peng_coordinate_2016}. For these reasons, we are interested in \emph{asynchronous} distributed algorithms.
A second challenge is that of \emph{unreliable communications} between the agents, which may be subject to packet losses or quantization errors, especially when wireless transmission is employed \cite{ye_decentralized_2022}. It is therefore important to characterize the resilience of distributed algorithms to imperfect communications.
Finally, in applications such as learning, the large size of the data set stored by an agent implies that evaluating the local gradient may be too computationally expensive. For this reason, \emph{stochastic gradients} are often employed to approximate the local gradient over a subset of the data, which introduces additive errors \cite{li_variance_2022,lei_distributed_2022,hu_push-lsvrg-up_2023}.

Let us now turn back to the distributed algorithms discussed above; among them, ADMM has arguably proven to be the most robust. In particular, it guarantees exact, almost sure convergence in the presence of asynchrony and packet losses \cite{peng_arock_2016,bastianello2020asynchronous,bianchi_coordinate_2016,chang_asynchronous_2016}, and can reach an approximate solution in the presence of additive communication errors \cite{majzoobi_analysis_2019,chang_proximal_2016,elgabli_q-gadmm_2021,bastianello2022admm}.
The drawback of ADMM is that the agents need to solve a local minimization problem to update their state, which in general is more computationally expensive than the gradient evaluations of GT.
Gradient Tracking methods also exhibit convergence in the presence of asynchrony and packet losses \cite{bof_multiagent_2019,tian_achieving_2020,carnevale2023triggered}. However, as highlighted in \cite{pu_robust_2020,bin2019system,notarnicola2023gradient,bin2022stability,wang_gradient-tracking_2022}, they are not robust to additive errors due to, e.g., unreliable communications.
Indeed, GT algorithms are designed to reconstruct the global gradient $\nabla \sum_{i = 1}^N f_i$ using local gradients $\nabla f_i$, which they accomplish employing a dynamic average consensus block (see~\cite{zhu2010discrete,kia2019tutorial}). The average consensus, though, introduces a \emph{marginally stable dynamics}~\cite{bin2022stability}, which causes the lack of robustness to errors.

The objective of this paper then is to propose a novel gradient tracking algorithm that is provably robust to asynchrony and unreliable communications. Our starting point is the insight that the dynamic average consensus block in GT algorithms can be \emph{replaced by a different, robust consensus protocol}, as done with the ratio (or push-sum) consensus in~\cite{bof_is_2018,tian_achieving_2020}.
Specifically, our algorithm uses an ADMM-based dynamic consensus protocol, which was shown to be robust to both asynchrony and unreliable communications \cite{bastianello2022admm}, differently from ratio consensus. The protocol is derived by formulating dynamic consensus as an \emph{online, quadratic problem} and applying the robust distributed ADMM to it~\cite{bastianello2022admm}.
Since the consensus problem is quadratic, the agents' updates have a closed form and no minimization needs to be solved, avoiding the larger computational footprint of ADMM for general problems. 
Our contribution then brings together the best of both Gradient Tracking, with its lightweight computations, and ADMM, with its robustness.

In the following, we summarize our contributions.
We propose a novel distributed gradient tracking method that employs an ADMM-based dynamic consensus protocol. The algorithm then inherits the robustness properties of ADMM to both asynchrony and unreliable communications.
We analyze the convergence of the proposed algorithm -- for synchronous and reliable networks -- by interpreting it as a \emph{singularly perturbed} system given by the interconnection between (i) a slow subsystem given by an approximation of the gradient descent and a consensus step, and (ii) a fast one related to the states of the ADMM-based consensus protocol. With these tools, we are able to show linear convergence to the optimal solution of the consensus optimization problem.
Then, this result allows us to analyze a robust version of the algorithm tailored for networks with asynchrony and packet losses. 
\GC{In this time-varying setup}, we leverage both singular perturbations and \emph{averaging theory} to show that the linear convergence of the scheme is preserved.
Moreover, we are further able to prove that the proposed schemes exhibit Input-to-State Stability (ISS) with respect to additive errors, due, e.g., to unreliable communications or the use of stochastic gradients.
Finally, we showcase the performance of the algorithms in numerical simulations for quadratic and logistic regression problems over random networks. We compare its performance with a standard and a robust version of the GT, which lack the same robustness properties.
Preliminary results related to this paper appeared in~\cite{carnevale2023distributed}.
However, the preliminary work~\cite{carnevale2023distributed} does not consider the scenario over asynchronous and unreliable networks and does not analyze the algorithm's robustness. %
Moreover, in~\cite{carnevale2023distributed}, the convergence proof in the case of ideal networks is omitted.

The paper is organized as follows.
Section~\ref{sec:setup} introduces the problem setup.
In Section~\ref{sec:algo_design}, we design the novel distributed algorithm and provide its convergence properties.
Section~\ref{sec:analysis_algo} is devoted to analyzing the algorithm.
In Section~\ref{sec:asynchronous_and_lossy}, we present a robust version of the algorithm to address the scenario with imperfect networks, while Section~\ref{sec:analysis_ralgo} is devoted to analyzing this second algorithm.
Section~\ref{sec:robustness} characterizes the robustness of the proposed algorithms.
Finally, Section~\ref{sec:numerical_simulations} compares our methods and \GT/ algorithms. %

\paragraph*{Notation}
The identity matrix in $\R^{m\times m}$ is $I_m$. 
The vector of $N$ ones and $N$ zeroes are denoted by $1_N$ and $0_N$, respectively, while
$\oneNn := 1_N \otimes I_\n$ and $\mathbf{0}_{N,\n} := 0_N \otimes I_\n$ with $\otimes$ being the Kronecker product. 
For a finite set $S$, we denote by $|S|$ its cardinality.
The symbol $\col (v_i)_{i \in \until{N}}$ denotes the vertical concatenation of the column vectors $v_1, \dots, v_N$.
We denote as $\blkdiag(M_1, \dots, M_N) \in \R^{\sum_{i=1}^N n_i}$ the block diagonal matrix whose $i$-th block is given by $M_i \in \R^{n_i \times n_i}$.

	\section{Problem Description and Preliminaries}
	\label{sec:setup}
	
	\subsection{Problem Setup}

	We consider a network of $N$ agents that aim to solve %
	\begin{align}\label{eq:problem}
		\min_{\x \in \R^\n} \sum_{i=1}^N f_i(\x),
	\end{align}
	where $\x \in \R^\n$ is the (common) decision variable, while $f_i: \R^\n \to \R$ is the objective function of agent $i \in \until{N}$.
	In the following, we will also use the function $f : \R^\n \to \R$ defined as $f(\x) := \sum_{i=1}^{N} f_i(\x)$.
	Our goal is to design an algorithm to solve~\eqref{eq:problem} in a distributed way, namely with update laws implementable over a network of agents using only (i) local information and (ii) neighboring communication.
	Indeed, we consider a network of agents communicating according to an undirected graph $\cG = (\until{N}, \cE)$, with $\cE \subset \until{N}\times \until{N}$ such that $i$ and $j$ can exchange information only if $(i,j)\in\cE$.  
	The set of neighbors of $i$ is
	$\cN_i := \{j \in \until{N} \mid (i,j) \in \cE\}$, while its degree is $\dii:= |\cN_i|$ and $\dd := \sum_{i=1}^N \dii$. 
	Notice that it holds $\dd = 2|\cE|$. %
	
	The following assumptions formalize the considered setup.
	\begin{assumption}\label{ass:network}[Graph]
		The graph $\cG$ is connected. \oprocend%
	\end{assumption}
	\begin{assumption}\label{ass:strong_and_lipschitz}[Objective functions]
		The objective function $f$ is $\str$-strongly convex, while the gradients $\nabla f_i$ are $\lip$-Lipschitz continuous for all $i \in \until{N}$.\oprocend
	\end{assumption}
	We notice that Assumption~\ref{ass:strong_and_lipschitz} ensures that problem~\eqref{eq:problem} has a unique minimizer and we denote it as $\xstar \in \R^\n$.

	\subsection{Fundamentals of Singularly Perturbed Systems}
    \label{sec:sp}

In the next, we will propose two novel distributed algorithms and, to assess their convergence properties, we will resort to a system-theoretic perspective based on \emph{singular perturbation} theory.
To this end, we provide below a generic convergence result extending~\cite[Theorem~II.5]{carnevale2022tracking} for the time-varying case.
Indeed, we will also consider the case in which the agents are asynchronous and their communication is subject to packet losses and, thus, the result~\cite[Theorem~II.5]{carnevale2022tracking} does not suit our purposes since it only considers time-invariant dynamics.
\begin{theorem}[\textup{Global exponential stability for time-varying singularly perturbed systems}]\label{th:theorem_generic}
	Consider the system
	\begin{subequations}\label{eq:interconnected_system_generic}
		\begin{align}
			\xtp &= \xt + \pr f(\xt,\zt,\iter)\label{eq:slow_system_generic}
			\\
			\ztp &= g(\zt,\xt,\iter),\label{eq:fast_system_generic}
		\end{align}
	\end{subequations}	
	with $\xt \in \mathcal{D} \subseteq \R^n$, $\zt \in \R^m$, $\map{f}{\mathcal{D} \times \R^m \times \N}{\cD}$, $\map{g}{\R^m \times \R^n \times \N}{\R^m}$, $\pr > 0$. 
	Assume that $f$ and $g$ are Lipschitz continuous uniformly in $\iter$ with respect to $\xt$ and $\zt$ with Lipschitz constants $L_f, L_g > 0$, respectively.
	Assume that there exists $\zeq: \R^n \to \R^m$ such that for all $\x \in \cD$ it holds
	\begin{align*}
		0 &= \pr f(0, \zeq(0),
		\iter)
		\\
		\zeq(\x) &= g(\zeq(\x),\x,\iter),
	\end{align*}
	for all $\iter \in \N$ and with $\zeq$ being Lipschitz continuous with Lipschitz constant $L_{\zeq} > 0$. 
	Let
	\begin{equation}\label{eq:reduced_system_generic}
		\xtp = \xt + \pr  f(\xt,\zeq(\xt),\iter)
	\end{equation}
	be the reduced system and
	\begin{equation}\label{eq:boundary_layer_system_generic}
		\tztp = g(\tzt + \zeq(\x),\x,\iter) - \zeq(\x)
	\end{equation}
	be the boundary layer system with $\tzt \in \R^m$. 
	
	Assume that there exist a continuous function $U: \R^m \times \N \to \R$ and $b_1, b_2, b_3, b_4 > 0$ such that
	\begin{subequations}\label{eq:U_generic}
		\begin{align}
			&b_1 \norm{\tz}^2 \leq U(\tz,\iter) \leq b_2\norm{\tz}^2 \label{eq:U_first_bound_generic}
			\\
			&U(g(\tz  \!+\!  \zeq(\x),\x,\iter) \!-\! \zeq(\x),\iterp) \!-\!  U(\tz,\iter)  
			\!\leq\!  
			\! - \! b_3 \! \norm{\tz}^2\!
			\label{eq:U_minus_generic}
			\\
			&|U( \tz_1,\iter)-U( \tz_2,\iter)|\leq b_4\norm{\tz_1- \tz_2}\left(\norm{\tz_1} + \norm{\tz_2}\right)
			,\label{eq:U_bound_generic}
		\end{align}
	\end{subequations}
	for all $\tz,  \tz_1,  \tz_2 \in \R^m$, $\x \in \R^n$, $\iter \in N$.
	Further, assume there exist a continuous function $W:\mathcal{D} \times \N \to \R$ and $\bar{\pr}_1 > 0$ such that, for all $\pr \in (0,\bar{\pr}_1)$, there exist $c_1, c_2, c_3, c_4 > 0$ such that%
	\begin{subequations}\label{eq:W_generic}
		\begin{align}
			&c_1 \norm{\x}^2 \leq W(\x,\iter) \leq c_2\norm{\x}^2\label{eq:W_first_bound_generic}
			\\
			&W(\x + \pr  f(\x,\zeq(\x),\iter),\iter)  -  W(\x,\iter) \leq - c_3\pr\norm{\x}^2\label{eq:W_minus_generic}
			\\
			&|W(\x_1,\iter)-W(\x_2,\iter)|\leq c_4\norm{\x_1-\x_2}\norm{\x_1}
			\notag\\
			&\hspace{.4cm}
			+ c_4\norm{\x_1-\x_2}\norm{\x_2},\label{eq:W_bound_generic}
		\end{align}
	\end{subequations}
	for all $\x, \x_1, \x_2, \x_3 \in \mathcal{D}$, $\iter \in \N$.

	Then, there exist $\bar{\pr} \in (0,\bar{\pr}_1)$, $\kappa_1 >0$, and $\kappa_2 > 0$ such that, for all $\pr \in (0,\bar{\pr})$, it holds
	\begin{align*}
		\norm{\begin{bmatrix}
				\GC{\xt}
                \\
				\zt - \zeq(\xt)
		\end{bmatrix}} \leq \kappa_1\norm{\begin{bmatrix}
				\x^0\\
				\z^0 - \zeq(\x^0)
		\end{bmatrix}} e^{-\kappa_2t},
	\end{align*}
	for all $(\x^0,\z^0) \in \mathcal{D} \times \R^m$.
\end{theorem}
The proof of Theorem~\ref{th:theorem_generic} is given in Appendix~\ref{sec:proof_generic}.

	\section{\algo/}
	\label{sec:algo_design}

    \GC{

    \subsection{\GC{From average consensus to consensus ADMM}}
	\label{sec:from_avg_to_admm}

    Let $\xit \in \R^\n$ be the estimate of the solution to problem~\eqref{eq:problem} maintained by agent $i$ at iteration $\iter \in \N$.
    Then, we should update it with the twofold purpose of removing (i) the consensus error with respect to the other agents' estimates, and (ii) the optimality error related to problem~\eqref{eq:problem}, namely 
	\begin{align}
		\xitp = \xit + \pr\bigg(\frac{1}{N}\sum_{j=1}^N\xjt - \xit\bigg) - \pr\frac{\step}{N} \sum_{j=1}^N \nabla f_j(\xjt),\label{eq:desired_control_law}
	\end{align}
	where $\pr,\step > 0$ are tuning parameters.
	However, in a distributed setting, agent $i$ cannot access $\frac{1}{N}\sum_{j=1}^N\xjt$ and $\frac{1}{N}\sum_{j=1}^N \nabla f_j(\xjt)$.
	Hence, one may modify~\eqref{eq:desired_control_law} by employing two auxiliary variables $\yit, \sit \in \R^\n$ aimed at reconstructing $\frac{1}{N}\sum_{j=1}^N \xjt$ and $\frac{1}{N}\sum_{j=1}^N \nabla f_j(\xjt)$, respectively, by running average consensus.
    Thus, the overall algorithm would read as
    \begin{subequations}\label{eq:gt-with-consensus}
    \begin{align}
        \xitp &= \xit + \pr(\yit - \xit) - \pr\step\sit
        \label{eq:gt-with-consensus-x}
        \\
        \begin{bmatrix}\yitp
        \\
        \sitp
        \end{bmatrix}&= \sum_{j \in \mathcal{N}_i} w_{ij} \begin{bmatrix}\yjt
        \\
        \sjt
        \end{bmatrix} + \begin{bmatrix}\xitp
        \\
        \nabla f_i(\xitp)
        \end{bmatrix} - \begin{bmatrix}\xit
        \\
        \nabla f_i(\xit)
        \end{bmatrix},
        \label{eq:gt-with-consensus-y}
    \end{align}
    \end{subequations}
    where each weight $w_{ij} \ge 0$ is the $(i,j)$-entry of a weighted adjacency matrix $\cW \in\R^{N\times N}$ matching the graph $\cG$, i.e., $w_{ij} >0$ whenever $(j,i)\in \cE$ and $w_{ij} =0$ otherwise.
    In ideal conditions, it is possible to show that~\eqref{eq:gt-with-consensus} converges to a solution to problem~\eqref{eq:problem}.
    However, the (dynamic) average consensus of~\eqref{eq:gt-with-consensus-y}
    is not suited to more challenging scenarios including asynchronous activation of the agents and communication losses~\cite{hadjicostis_robust_2016,bof_average_2017}.
    Thus, we need to replace average consensus with a robust consensus protocol.

    Different works have explored the use of \emph{push-sum} (or \emph{ratio}) consensus~\cite{bof_multiagent_2019,tian_achieving_2020} for tracking since this protocol can be modified to converge even with asynchrony and packet losses (cf.~\cite{hadjicostis_robust_2016,bof_average_2017}).
    The key aspect of push-sum is to %
    ensure the so-called \emph{mass preservation} even in the event of packet losses by means of additional pairs of variables for each edge.

    In this paper, we apply an alternative consensus protocol to robustly track the global quantities in the presence of asynchrony, packet losses, and inexactnesses:
    \emph{consensus ADMM}.
    Indeed, differently from consensus ADMM, both average consensus and push-sum need mass preservation, which, however, is not verified when inexact communications and computations occur (due to, e.g., quantization).

    Consensus ADMM, similarly to push-sum protocols, requires additional variables associated to each edge to ensure convergence in asynchronous and lossy scenarios.
    Table~\ref{tab:consensus-protocols} summarizes the preceding discussion.
    We conclude this section by noting that, in the presence of critical memory constraints, one may adopt a less dense graph topology contained within the ``original'' one to reduce the memory burden of the scheme.
    For instance, in the ring topology, it holds $\degree_i = 2$ for all $i$.
    \begin{table}[H]
    \centering
    \caption{Comparison of various dynamic consensus protocols.}
    \label{tab:consensus-protocols}
    \begin{tabular}{cccc}
    \hline
    Protocol & Asynch. \& lossy & Inexactness & \# local variables \\
    \hline
    Average & \xmark & \xmark & $1$ \\
    Push-sum & \cmark & \xmark & $4 + 2 d_i$ \\
    ADMM & \cmark & \cmark & $d_i$ \\
    \hline
    \end{tabular}
    \end{table}
    }

    \subsection{\algo/: Algorithm Design and Convergence Properties}

    \GC{
    The previous section showed the need for consensus ADMM to robustly tracking $\frac{1}{N}\sum_{j=1}^N \xjt$ and $\frac{1}{N}\sum_{j=1}^N \nabla f_j(\xjt)$ in a distributed manner.
    The key to applying ADMM for a dynamic consensus problem is to reformulate it as the \emph{online optimization} problem \cite{bastianello2022admm}}%
	\begin{align}\label{eq:cns_as_opt}
		\begin{split}
			\min_{
				\substack{
					(\y_1,\dots,\y_N) \in \R^{N\n}
					\\
					(\s_1,\dots,\s_N) \in \R^{N\n}
				}
			}
			&\sum_{i=1}^N \git(\y_i,\s_i)
			\\
			\text{s.t.:} \: \: &\begin{bmatrix}\y_i\\
				\s_i\end{bmatrix} = \begin{bmatrix}\y_j\\\s_j\end{bmatrix}  \: \forall (i,j)\in \cE,
		\end{split}
	\end{align}
	where, for all $i \in \until{N}$, $\git : \R^{\n} \times \R^{\n} \to \R$ reads as 
	\begin{align*}
		\git(\y_i,\s_i) = \frac{1}{2}\norm{\y_i - \xit}^2 + \frac{1}{2}\norm{\s_i - \nabla f_i(\xit)}^2.
	\end{align*}
	Indeed, if the graph $\cG$ is connected, then the (unique) optimal solution to problem~\eqref{eq:cns_as_opt}, say it $(\ystart,\sstart) \in \R^{2N\n}$, reads as $(\ystart,\sstart) = (\oneNn\frac{1}{N}\sum_{j=1}^N \xjt,\oneNn\frac{1}{N}\sum_{j=1}^N \nabla f_j(\xjt))$~\cite{bastianello2022admm}.
	From this observation, we design the updates of $\yit$ and $\sit$ by resorting to the distributed ADMM proposed in~\cite{bastianello2020asynchronous} (see Appendix~\ref{sec:admm} for a description of this method in a generic framework).
	Hence, each agent $i$ maintains an additional variable $\zijt \in \R^{2\n}$ for each neighbor $j \in \cN_i$ and implements 
	\begin{subequations}
		\begin{align}
			\begin{bmatrix}
				\yit
				\\
				\sit
			\end{bmatrix}
			&= \argmin_{\substack{\y_i \in \R^{\n} 
					\\
					\s_i \in \R^\n
				}
			}
			\Bigg\{\git(\y_i,\s_i) - \begin{bmatrix} \y_i\T& \s_i\T\end{bmatrix} \sum_{j \in \cN_i}\zijt 
			\notag\\
			&\hspace{2cm}+ \frac{\rho \dii}{2}\left(\norm{\y_i}^2 + \norm{\s_i}^2\right)\bigg\}
			\\
			\zijtp &= (1 - \alpha)\zijt + \alpha\left(-\zjit + 2\rho\begin{bmatrix}
				\yjt
				\\
				\sjt
			\end{bmatrix}\right),\label{eq:zij_update_pre}
		\end{align}
	\end{subequations}
	with $\rho > 0$ and $\alpha \in (0,1)$.
	Being $\git$ quadratic, the above updates are equivalent to the closed form
	\begin{subequations}\label{eq:ADMM_update}
		\begin{align}
			\begin{bmatrix}
				\yit
				\\
				\sit
			\end{bmatrix} &= \frac{1}{1+\rho \dii}\left(\begin{bmatrix}\xit\\
				\nabla f_i(\xit)\end{bmatrix} + \sum_{j\in\cN_i} \zijt\right)
			\\
			\zijtp &= (1 - \alpha)\zijt + \alpha\msgjit,\label{eq:zij_update}
		\end{align}
	\end{subequations}
	in which we also introduced $\msgjit \in \R^{2\n}$ to denote the message from agent $j$ needed by agent $i$ to perform~\eqref{eq:zij_update_pre}, namely
	\begin{align}
		\msgjit := -\zjit + 2\rho\begin{bmatrix}
			\yjt 
			\\
			\sjt
		\end{bmatrix},\label{eq:msgjit}
	\end{align}
	for all $i \in \until{N}$ and $j \in \cN_i$.
    \GC{By replacing~\eqref{eq:gt-with-consensus-y}
    with~\eqref{eq:ADMM_update}, we get the whole distributed protocol reported in Algorithm~\ref{algo:algo} that we name \algo/.}
	\begin{algorithm}[H]
		\begin{algorithmic}
			\State \textbf{Initialization}: $\x_i^0 \in \R^{\n}$, \GC{$\z_{ij}^0 \in \R^{2\n}$ for all $j \in \cN_i$.}
			\For{$\iter=0, 1, \dots$}
			\vspace{.1cm}
				\State $\begin{bmatrix}
					\yit
					\\
					\sit
				\end{bmatrix} = \frac{1}{1+\rho \dii}\left(\begin{bmatrix}\xit\\\nabla f_i(\xit)\end{bmatrix} + \sum_{j\in\cN_i} \zijt\right)$
				\vspace{.1cm}
				\State $\xitp = \xit + \pr(\yit - \xit) - \pr\step\sit$
				\For{$j \in \cN_i$}
					\State $\msgijt = -\zijt + 2\rho\begin{bmatrix}
						\yit
						\\
						\sit
					\end{bmatrix}$
					\State transmit $\msgijt$ to $j$ and receive $\msgjit$ to $j$
					\State $\zijtp = (1-\alpha)\zijt + \alpha\msgjit$
				\EndFor
			\EndFor
		\end{algorithmic}
		\caption{\algo/ (Agent $i$)}
		\label{algo:algo}
	\end{algorithm}
	We remark that Algorithm~\ref{algo:algo} can be implemented in a fully-distributed fashion since it only requires neighboring communication and local variables. 
    \GC{As for the memory required by the execution of Algorithm~\ref{algo:algo}, agent $i$ needs to store the solution estimate $\xit \in \R^{\n}$ and the auxiliary variable $\zit := \col(\zijt)_{j\in\cN_i}\in \R^{2\n\degree_i}$ stacking all the variables $\zijt$ and, thus, vectors with an overall size $(1+\degree_i)\n$.}
	The next theorem states the convergence features of Algorithm~\ref{algo:algo}.
	\begin{theorem}\label{th:convergence}
		Consider Algorithm~\ref{algo:algo} and let Assumptions~\ref{ass:network} and~\ref{ass:strong_and_lipschitz} hold.
		Then, there exist $\bar{\pr}, \bar{\step}, c_1, c_2 > 0$ such that, for all $\rho >0$, $\alpha \in (0,1)$, $\pr \in (0,\bar{\pr})$, \GC{$\step \in (0,\tfrac{2\str}{NL^2})$}, $(\x_i^0,\z_i^0) \in \R^\n \times \R^{2\n\degree_i}$, for all $i \in \until{N}$, it holds
		\begin{align*}
			\norm{\xit - \xstar} \leq c_1\exp(-c_2\iter).\eqoprocend
		\end{align*}
	\end{theorem}
	Theorem~\ref{th:convergence} is proved in Section~\ref{sec:proof_algo}.
	In detail, with tools from system theory, we will show exponential stability of $(\oneNn\xstar,\z^\star)$ for the aggregate form of \algo/, where the quantity $\z^\star \in \R^{2\n\dd}$ will be specified later.
    \GC{
    \begin{remark}
        \algo/ is characterized by four tuning parameters: $\pr, \step, \alpha, \rho$. 
        Theorem~\ref{th:convergence} ensures linear convergence of \algo/ with any $\alpha \in (0,1)$, $\rho > 0$, and sufficiently small values of $\pr$ and $\step$. 
        In particular, being $\step$ a step-size parameter, we are able to provide an exact bound resembling the one prescribed for (centralized) gradient descent.
        Instead, as it will become more clear in Section~\ref{sec:analysis_algo}, $\pr$ acts as a timescale parameter (see Section~\ref{sec:sp}) and its bound depends on many setup parameters (such as, e.g., network connectivity).
        For practical tuning, it is advisable to start with small values of $\pr$ and increase them as much as possible.
    \end{remark}
    }

	\section{\algo/:\\ Stability Analysis}
	\label{sec:analysis_algo}
	
	Here, we provide the analysis of Algorithm~\ref{algo:algo}. 
	First, in Section~\ref{sec:system reformulation}, we interpret the aggregate form of \algo/ as a \emph{singularly perturbed} system, i.e, the interconnection between a slow and fast subsystem.
	Then, in Section~\ref{sec:bl_algo} and~\ref{sec:rs_algo}, respectively, we separately study the identified subsystems by relying on Lyapunov theory.
	Finally, in Section~\ref{sec:proof_algo}, we use the results collected in the previous steps to prove Theorem~\ref{th:convergence}.
    \GC{Throughout the whole section, the assumptions of Theorem~\ref{th:convergence} hold true.}

	\subsection{\algo/ as a Singularly Perturbed System}
	\label{sec:system reformulation}

	\GC{We start by providing the aggregate formulation of \algo/.}
	To this end, let us introduce the permutation matrix $P \in \R^{2\n\dd \times 2\n\dd}$ that swaps the $ij$-th element with the $ji$-th one, and the matrices $\Ax \in \R^{2\n\dd \times N\n}$, \GC{$\An \in \R^{2\n\dd \times N\n}$}, $A \in \R^{2\n\dd \times N\n}$, $H \in \R^{N\n \times N\n}$, and $\cH \in \R^{2N\n \times 2N\n}$ %
 given by 
	\begin{align*}
			\Ax &:=\begin{bmatrix}
					\one_{\degree_1,\n}
							\\
							0_{\degree_1,\n}
					\\
					&\diagentry{\xddots}
					\\
					&&
					\one_{\degree_N,\n}
					\\
					&&0_{\degree_N,\n}
				\end{bmatrix} 
		\!\!,\hspace{.08cm}
			\GC{\An :=} \begin{bmatrix}
					0_{\degree_1,\n}
					\\
					\one_{\degree_1,\n}
					\\
					&\diagentry{\xddots}
					\\
					&&0_{\degree_N,\n}
					\\		
					&&\one_{\degree_N,\n}
				\end{bmatrix}
				\\
				A &:=  \begin{bmatrix}
					\one_{\degree_1,2\n}
					\\
					&\diagentry{\xddots}
					\\
					&&\one_{\degree_N,2\n}
				\end{bmatrix}
				\\
				H &:= \begin{bmatrix}
					\frac{1}{1+\rho\degree_1}I_{\n}
							\\
							&\diagentry{\xddots}
							\\
							&&
							\frac{1}{1+\rho\degree_N}I_{\n}
				\end{bmatrix}
				\\
				\cH &:= \begin{bmatrix}
					\frac{1}{1+\rho\degree_1}I_{2\n}
							\\
							&\diagentry{\xddots}
							\\
							&&
							\frac{1}{1+\rho\degree_N}I_{2\n}
				\end{bmatrix}.
	\end{align*}
	Then, we define the stacking vectors $\xt := \col(\xt_1,\dots,\xt_N) \in \R^{N\n}$ and $\zt := \col(\zt_1,\dots,\zt_N) \in \R^{2\n\dd}$.
  Then, the aggregate formulation of \algo/ reads as
	\begin{subequations}\label{eq:algo_aggregate_form}
		\begin{align}
			\xtp &= \xt + \pr \left(H\left(\xt + \Ax\T\zt\right) -\xt\right)
			\notag\\
			&\hspace{.4cm}
			- \pr \step H  \left(\G(\xt) + \An\T\zt\right)\label{eq:algob_aggregate_form_x}
			\\
			\ztp &= \zt - \alpha(I + P - 2\rho PA\cH A\T)\zt 
			\notag\\
			&\hspace{.4cm}
			+ 2\alpha\rho PA\cH  \vv(\xt),\label{eq:algob_aggregate_form_z}
		\end{align}
	\end{subequations}
	where we introduced the operators $\G: \R^{N\n} \to \R^{N\n}$ and $\vv: \R^{N\n} \to \R^{2N\n}$ that, given any $\x := \col(\x_1,\dots,\x_N) \in \R^{N\n}$ with $\x_i \in \R^\n$ for all $i \in \until{N}$, are defined as $G(\x) := \col(\nabla f_1(\x_1),\dots,\nabla f_N(\x_N))$ and $\vv(\x) := \col(\x_1,\nabla f_1(\x_1),\dots, \x_N,\nabla f_N(\x_N))$.
	Fig.~\ref{fig:block_diagram} reports a block diagram graphically describing~\eqref{eq:algo_aggregate_form}.
	\begin{figure}[H]
		\includegraphics[width=\columnwidth]{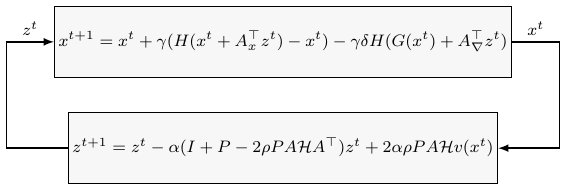}
		\caption{Block diagram representing~\eqref{eq:algo_aggregate_form}.}
		\label{fig:block_diagram}
	\end{figure}
	From the discussion in~\cite{bastianello2020asynchronous}, we deduce that the matrix $(I + P - 2\rho P A\cH A\T)$ has some eigenvalues in $0$ that are all semi-simple.
	Next, we introduce a decomposition to express $\zt$ according to a basis of the subspace corresponding to the kernel of $(I + P - 2\rho P A\cH A\T)$.
	To this end, let $\bb \in \N$ be the dimension of the subspace $\cS$ spanned by the eigenvectors of $(I + P - 2\rho P A\cH A\T)$ associated to $0$, $\B \in \R^{2\n\dd \times \bb}$ be the matrix whose columns represent an orthonormal basis of $\cS$, and $\M \in \R^{2\n\dd \times \m}$ be the matrix such that $\M\T M = I$ and $\B\T \M = 0$, with $\m := 2\n\dd - \bb$.
	The next lemma highlights some useful properties of the matrices $\B$ and $\M$.
	\begin{lemma}\label{lemma:results}
		Consider the matrices $\B$ and $\M$.
		Then, it holds%
		\begin{subequations}\label{eq:results}
		\begin{align}
			\Ax\T \B &= 0, \quad
			\An\T \B = 0\label{eq:results_1}
			\\
			\B\T PA &= 0\label{eq:B_P_A}
			\\
			\B\T (I \! + \! P \! - \! 2\rho PA\cH A\T)&= 0.\label{eq:results_2}
		\end{align}
	\end{subequations}
	\end{lemma}
	\begin{proof}
		By invoking~\cite[Lemma~2]{bastianello2020asynchronous}, we claim that 
		\begin{align}\label{eq:kernels}
			\ker(I + P - 2\rho PA\cH A\T) \subset \ker(A\T).
		\end{align}
		Further, by construction of $\Ax$ and $\An$, the results $\ker(A) \subseteq \ker(\Ax)$ and $\ker(A) \subseteq \ker(\An)$ hold true.
		Thus, the proof of~\eqref{eq:results_1} follows.
		In order to prove~\eqref{eq:B_P_A}, we note that~\eqref{eq:kernels} implies $(I + P)\B = 0$.
		Hence, since $P = P\T$, it holds 
		\begin{align}\label{eq:B_I_P}
			\B\T(I + P) = 0.
		\end{align}
		Further, the result~\eqref{eq:kernels} also implies 
		\begin{align}\label{eq:B_A}
			\B\T A = 0.
		\end{align}
		Moreover, by construction, it holds
		\begin{align}\label{eq:P_A}
			\ker(A\T P) \subseteq \ker (A\T ).
		\end{align}
		The result~\eqref{eq:B_P_A} follows by combining~\eqref{eq:B_A},~\eqref{eq:P_A} and the fact that $P = P\T$.
		Finally, as for~\eqref{eq:results_2}, it follows by combining~\eqref{eq:B_P_A} and~\eqref{eq:B_I_P}.
	\end{proof}
	Now, let us introduce the novel coordinates $\bz \in \R^\bb$ and $\pz \in \R^{\m}$ defined as
	\begin{align}\label{eq:z_perp_z_bar}
		\bz := \B\T\z, \quad \pz := \M\T\z.
	\end{align}
	Then, by invoking Lemma~\ref{lemma:results}, we use the coordinates~\eqref{eq:z_perp_z_bar} to equivalently rewrite system~\eqref{eq:algo_aggregate_form} as
	\begin{subequations}\label{eq:algo_transformed}
		\begin{align}
			\xtp &= \xt + \pr \left(H\left(\xt + \Ax\T\M \pzt\right) -\xt\right)
			\notag\\
			&\hspace{.4cm}
			- \pr \step H  \left(\G(\xt) + \An\T\M \pzt\right)
			\label{eq:algo_transformed_px}
			\\
			\bztp &= \bzt \label{eq:algo_transformed_bz}
			\\
			\pztp &= \pzt - \alpha\M\T (I + P - 2\rho PA\cH A\T)\M \pzt 			\notag\\
			&\hspace{.4cm}
			+ 2\alpha\rho \M\T PA\cH \vv(\xt).
			\label{eq:algo_transformed_pz}
		\end{align}
	\end{subequations}
	Notably, the variable $\bzt$ does not affect the other updates of~\eqref{eq:algo_transformed} (and it holds $\bzt \equiv \bz^0$ for all $\iter\in\N$).
	Thus, by ignoring the variable $\bzt$, and introducing the error coordinate $\txt := \xt - \oneNn\xstar$, system~\eqref{eq:algo_transformed} is equivalent to
	\begin{subequations}\label{eq:sp_system}
		\begin{align}
			\txtp &= \txt + \pr \left(H\left(\xt + \Ax\T\M \pzt\right) -\xt\right)
			\notag\\
			&\hspace{.4cm}
			- \pr \step H  \left(\G(\xt) + \An\T\M \pzt\right)\label{eq:sp_system_slow}
			\\
			\pztp &= \pzt \! - \! \alpha F \pzt 			
			\! + \! 2\alpha\rho \M\T PA\cH \vv(\xt),\label{eq:sp_system_fast}
		\end{align}
	\end{subequations}
	where, with a slight abuse of notation, we used $\xt = \txt + \oneNn\xstar$ and introduced the matrix $\F \in \R^{\m \times \m}$ defined as 
	\begin{align}\label{eq:F}
		\F := \M\T (I + P - 2\rho PA\cH A\T)\M.
	\end{align}
	Next, we provide a crucial property about the matrix $I - \alpha\F$.
	\begin{lemma}\label{lemma:F}
		Consider the matrix $\F$ as defined in~\eqref{eq:F}.
		Then, the matrix $I - \alpha\F$ is Schur for all $\alpha \in (0,1)$.
	\end{lemma}
	\begin{proof}
		By using~\eqref{eq:results_2} and~\eqref{eq:F}, it holds 
		\begin{align*}
			&\begin{bmatrix}
				\B\T 
				\\
				\M\T 
			\end{bmatrix}(I -\alpha(I + P - 2\rho PA\cH A\T))\begin{bmatrix}
				\B
				&
				\M
			\end{bmatrix}
			\\
			&
			= \begin{bmatrix}
				I& 0
				\\
				0& I - \alpha\F
			\end{bmatrix}.
		\end{align*}
		The result~\cite[Lemma~2]{bastianello2022admm} ensures that all the eigenvalues of $I - \alpha(I + P - 2\rho PA\cH A\T)$ are inside the circle on the complex plane with center $1 - \alpha$ and radius $\alpha$.
		Thus, since we isolated the eigenvalues equal to $1$, the proof follows.
	\end{proof}
	We interpret~\eqref{eq:sp_system} as a singularly perturbed system in the form of~\eqref{eq:interconnected_system_generic}, i.e., the interconnection between the slow subsystem~\eqref{eq:sp_system_slow}  and the fast one~\eqref{eq:sp_system_fast}.
	Indeed, the fast scheme~\eqref{eq:sp_system_fast} has an equilibrium parametrized in the slow state through the function $\pzeq: \R^{N\n} \to \R^{\m}$ defined as 
	\begin{align}
		\pzeq(\txt) &:=
		2\rho\F\inv \M\T PA\cH  \vv(\txt + \oneNn\xstar)
		.\label{eq:equilibrium}
	\end{align}
	Moreover, given any $\tx \in \R^{N\n}$ and the corresponding $\x = \tx + \oneNn\xstar$, it is possible to show that 
		\begin{subequations}\label{eq:results_equilibrium}
			\begin{align}
				H \Ax\T \M \pzeq(\tx) &= \frac{\oneNn\oneNn\T}{N}\x - H \x
				\\
				H \An\T \M \pzeq(\tx) &=\frac{\oneNn\oneNn\T}{N}\G(\x) -H \G(\x).
			\end{align}
		\end{subequations}

	\subsection{Boundary Layer System}
	\label{sec:bl_algo}

	In this section, we will study the so-called boundary layer system associated to~\eqref{eq:sp_system}.
	Therefore, we consider an arbitrarily fixed $\txt \equiv \tx \in \R^{N\n}$ for all $\iter \in \N$ and accordingly rewrite~\eqref{eq:sp_system_fast} using the error coordinate $\tpzt := \pzt - \pzeq(\tx)$.
	Hence, the definition of $\pzeq$ (cf.~\eqref{eq:equilibrium}) leads to
	\begin{align}\label{eq:bl_algo}
		\tpztp = \tpzt - \alpha\F\tpzt.
	\end{align}
	The next lemma provides a function $U$ satisfying~\eqref{eq:U_generic} for~\eqref{eq:bl_algo}.
	\begin{lemma}\label{lemma:bl_algo}
		Consider~\eqref{eq:bl_algo}. 
		Then, there exists $U: \R^{\m} \to \R$ such that the conditions~\eqref{eq:U_generic} are satisfied.
	\end{lemma}
	\begin{proof}
		We recall that the matrix $I - \alpha \F$ is Schur for all $\alpha \in (0,1)$ (cf.~Lemma~\ref{lemma:F}).
		Then, we arbitrarily choose $\Qtpz \in \R^{\m \times \m}$, $\Qtpz = \Qtpz\T > 0$ and, in view of the Schur property of $I - \alpha \F$, we claim that there exists $S_{\tpz} \in \R^{\m \times \m}$, $S_{\tpz} = S_{\tpz}\T > 0$ such that 
		\begin{align}\label{eq:Lyapunov}
			(I - \alpha \F)\T S_{\tpz}(I - \alpha\F) - S_{\tpz}  = - \Qtpz.
		\end{align}
		We then choose the candidate Lyapunov function $U(\tpz) = \tpz\T S_{\tpz}\tpz$.
		Conditions~\eqref{eq:U_first_bound_generic} and~\eqref{eq:U_bound_generic} are satisfied since the chosen $U$ is a quadratic positive definite function, while~\eqref{eq:U_minus_generic} follows by applying~\eqref{eq:Lyapunov}.
	\end{proof}
	
	\subsection{Reduced System}
	\label{sec:rs_algo}

	Now, we study the so-called reduced system associated to~\eqref{eq:sp_system}, i.e., the system obtained by considering $\pzt = \pzeq(\txt)$ into~\eqref{eq:sp_system_slow} for all $\iter\in\N$.
	By using~\eqref{eq:results_equilibrium}, such a system reads as 
	\begin{align}\label{eq:rs_algo}
		\txtp &= \txt - \pr\left(I - \frac{\oneNn\oneNn\T}{N}\right)\txt 
		\notag\\
		&\hspace{.4cm}
		- \pr\step\frac{\oneNn\oneNn\T}{N}\G(\txt + \oneNn\xstar).
	\end{align}
	The next lemma provides a Lyapunov function for system~\eqref{eq:rs_algo} satisfying the conditions required in~\eqref{eq:W_generic}.	
	\begin{lemma}\label{lemma:rs_algo}
		Consider~\eqref{eq:rs_algo}.
		Then, there exists $W: \R^{N\n} \to \R$ such that, for all $\pr \in (0,1)$ and $\step \in (0,2\str/(NL^2))$, $W$ satisfies~\eqref{eq:W_generic}.
	\end{lemma}
	\begin{proof}
	We start by introducing an additional change of variables to isolate (i) the optimality error and (ii) the consensus error related to the vector $\txt$.
	To this end, let $\rr \in \R^{N\n \times (N-1)\n}$ be the matrix such that its columns span the space orthogonal to that of $\oneNn$, that is, such that $\rr\T\oneNn = 0$ and $\rr\T\rr = I_{(N-1)\n}$.
	Then, let us introduce the new coordinates $\bx \in \R^{\n}$ and $\px \in \R^{(N-1)\n}$ defined as 
	\begin{align}\label{eq:change_x}
		\bx := \frac{\oneNn\T}{N}\tx, \quad \px := \rr\T \tx.
	\end{align}
	By using the coordinates~\eqref{eq:change_x}, $\oneNn \in \ker (I_{N\n} - \tfrac{\oneNn\oneNn\T}{N})$, and $\rr\T\tfrac{\oneNn\oneNn\T}{N} = 0$, we rewrite~\eqref{eq:rs_algo} as the cascade system
	\begin{subequations}\label{eq:rs_algo_explicit}
		\begin{align}
			\mutp &= \mut - \pr\GC{\step}\frac{\oneNn\T}{N}\G(\oneNn\mut + \rr\pxt + \oneNn\xstar)\label{eq:rs_algo_mu}
			\\
			\pxtp &= (1-\pr)\pxt.\label{eq:rs_algo_px}
		\end{align}
	\end{subequations}
			For the sake of compactness, let us introduce $\tilde{\G} : \R^\n \times \R^{(N-1)\n} \to \R^\n$ defined as
			\begin{align}
					\tilde{\G}\left(\bx,\px\right) &:= -\frac{\oneNn\T}{N}\G(\oneNn\bx + \rr\px + \oneNn\xstar) 
					\notag\\
					&\hspace{.4cm}
					+  \frac{\oneNn\T}{N}\G(\oneNn\bx + \oneNn\xstar).\label{eq:deltaG}
			\end{align}
			By using this notation and adding and subtracting $\frac{\pr\step}{N} \nabla f(\mut + \xstar)$ into~\eqref{eq:rs_algo_mu}, we compactly rewrite system~\eqref{eq:rs_algo_explicit} as
			\begin{subequations}\label{eq:sp_system_comapact}
				\begin{align}
					\mutp &= \mut - \frac{\pr\step}{N}\nabla f(\mut+\xstar) + \pr\step\tilde{\G}\left(\mut,\pxt\right) 
					\label{eq:sp_system_comapact_mu}
					\\
					\pxtp &= (1-\pr)\pxt.\label{eq:sp_system_comapact_px}
				\end{align}
			\end{subequations}
			Once this formulation is available, we consider the candidate Lyapunov function $W: \R^{N\n} \to \R$ defined as
			\begin{align}\label{eq:W_rs_algo}
				W(\tx) &:= 
				 \frac{1}{2}\tx\T \frac{\oneNn\oneNn\T}{N^2}\tx + \frac{m}{2}\tx\T \rr\rr\T\tx
				\notag\\
				&=\underbrace{\frac{1}{2}\norm{\bx}^2}_{:=\Vmu(\bx)} + \underbrace{\frac{m}{2}\norm{\px}^2}_{:=\Vpx(\px)},
			\end{align}
			where $m > 0$ will be fixed later.
			Being $W$ quadratic and positive definite, the conditions~\eqref{eq:W_first_bound_generic} and~\eqref{eq:W_bound_generic} are satisfied.
			To check condition~\eqref{eq:W_minus_generic}, we write $\Delta \Vmu(\mut) := \Vmu(\mutp) - \Vmu(\mut)$ along the trajectories of~\eqref{eq:sp_system_comapact_mu}, thus obtaining
			\begin{align}
				\Delta \Vmu(\mut) 
				&= -\frac{\pr\step}{N} (\mut)\T  \nabla f(\mut +\xstar)
				+\pr\step (\mut)\T \tilde{\G}(\mut,\pxt)
				\notag\\
				&\hspace{.4cm}
				+ \frac{\pr^2\step^2}{N} \nabla f(\mut +\xstar)\T \tilde{\G}(\mut,\pxt)
				\notag\\
				&
				\hspace{.4cm}
				+\frac{\pr^2\step^2}{2N^2} \norm{\nabla f(\mut +\xstar)}^2
				+\frac{\pr^2\step^2}{2}\norm{\tilde{\G}(\mut,\pxt)}^2\!\!.
				\label{eq:DeltaV_mu}
			\end{align}
			Since $\xstar$ is  the (unique) solution to problem~\eqref{eq:problem}, it holds $\nabla f(\xstar) = 0$ which allows us to write
			\begin{align}
				-(\mut)\T  \nabla f(\mut + \xstar) &= -(\mut)\T \left(\nabla f(\mut + \xstar)  - \nabla f(\xstar)\right)
				\notag\\
				&\stackrel{(a)}{\leq} 
				-\str\norm{\mut}^2,\label{eq:strong_convexity}
			\end{align}
			where in $(a)$ we use the $\str$-strong convexity of $f$ (cf. Assumption~\ref{ass:strong_and_lipschitz}).
			Then, since $\nabla f(\xstar) = 0$ and the gradients $\nabla f_i$ are Lipschitz continuous (cf. Assumption~\ref{ass:strong_and_lipschitz}), we get
			\begin{align}
				\norm{\nabla f(\mut + \xstar)} 
				&=  \norm{\nabla f(\mut + \xstar) - \nabla f(\xstar)}
				\notag\\
				&
				\stackrel{(a)}{\leq} L\GC{N}\norm{\mut}.\label{eq:bound_nabla_f}
			\end{align}
			Then, using the Lipschitz continuity of the gradients $\nabla f_i$ (cf. Assumption~\ref{ass:strong_and_lipschitz}) and $\norm{\rr} = 1$, we write the bound
			\begin{align}
				\norm{\tilde{\G}(\mut,\pxt)} \leq \tfrac{L}{\sqrt{N}}\norm{\pxt}.\label{eq:bound_delta_G}
			\end{align}
			\GC{Let us introduce $k_1 := L/\sqrt{N}$.}
   Hence, by using~\eqref{eq:strong_convexity},~\eqref{eq:bound_nabla_f}, and~\eqref{eq:bound_delta_G}, we bound~\eqref{eq:DeltaV_mu} as
			\begin{align}
				\Delta \Vmu(\mut) &\leq -\frac{\pr\step}{N} \str\norm{\mut}^2 
				+ \pr\step \GC{k_1}\norm{\mut}\norm{\pxt}
				\notag\\
				&\hspace{.4cm}
				+ \pr^2\step^2 \GC{k_1L}\norm{\mut}\norm{\pxt}
				+ \pr^2\step^2 \frac{L^2}{2}\norm{\mut}^2
				\notag\\
				&\hspace{.4cm}
				+ \pr^2\step^2 \frac{k_1^2}{2}\norm{\pxt}^2.
				\label{eq:DeltaV_mu_second}
			\end{align}
			Along the trajectories~\eqref{eq:sp_system_comapact_px}, we bound $\Delta \Vpx(\pxt) : = \Vpx(\pxtp) - \Vpx(\pxt)$ as
			\begin{align}
				\Delta \Vpx(\pxt) &\leq -m\pr\left(1 - \frac{\pr}{2}\right)\norm{\pxt}^2,\label{eq:DeltaV_px}
			\end{align} 
			Then, by defining $\Delta W(\txt) := \Delta \Vmu(\mut) + \Delta \Vpx(\pxt)$ and using~\eqref{eq:DeltaV_mu_second} and~\eqref{eq:DeltaV_px}, we get 
			\begin{align}
				\Delta W(\txt) 
				&\leq - \pr \begin{bmatrix}
					\norm{\mut}
					\\
					\norm{\pxt}
				\end{bmatrix}\T Q
				\begin{bmatrix}
					\norm{\mut}
					\\
					\norm{\pxt}
				\end{bmatrix} 
				,\label{eq:V_muV_px}	
			\end{align}
			where we introduced the matrix $Q = Q\T \in \R^{2\times 2}$ given by
			\begin{align*}
				Q := \begin{bmatrix}
					\step\str/N - \pr\step^2 L^2/2& -\step(\GC{k_1} + \pr\step \GC{k_1L})/2
					\\
					-\step(\GC{k_1} + \pr\step \GC{k_1L})/2& m(1 - \pr/2) - \pr\step^2 k_1^2/2
				\end{bmatrix}
			\end{align*}
			By Sylvester Criterion, it holds $Q > 0$ if and only if
			\begin{align}\label{eq:determinant_positive}
				\begin{cases}
					\step\str/N - \pr\step^2 L^2/2 > 0
					\\
					m > \bar{m}(\pr,\step),
				\end{cases}
			\end{align}
			where $\bar{m}(\pr,\step)$ is defined as
			\begin{align*}
				\bar{m}(\pr,\step) := \frac{\pr^2\step k_1^2/2(\str/N - \pr\step L^2/2) + \step(\GC{k_1} + \pr\step \GC{k_1L})^2/4}{\pr(\str/N - \pr\step L^2/2)(1 - \pr/2)}.
			\end{align*}
			Hence, we pick $\step \in (0,\tfrac{2\str}{NL^2})$, and, finally, choose $m > \bar{m}(\pr,\step)$.
			In this way,  since $\pr \in (0,1)$, both conditions~\eqref{eq:determinant_positive} are satisfied.
			Thus, we use $q > 0$ to denote the smallest eigenvalue of $Q$ and bound~\eqref{eq:V_muV_px} as 
			\begin{align*}
				\Delta W(\txt) 
				&\leq - \pr q(\norm{\mut}^2 
				+ \norm{\pxt}^2)
				\\
				&\leq -\pr q \tx\T \left(\frac{\oneNn\oneNn\T}{N^2} + \rr\rr\T\right)\txt
				,%
			\end{align*}
			which ensures that also condition~\eqref{eq:W_minus_generic} is satisfied since $\frac{\oneNn\oneNn\T}{N^2} + \rr\rr\T > 0$.
		\end{proof}

		\subsection{Proof of Theorem~\ref{th:convergence}}
		\label{sec:proof_algo}

		The proof of Theorem~\ref{th:convergence} is based on the exploitation of Theorem~\ref{th:theorem_generic}.
		In order to apply such a result, we need to (i) provide a function $U$ satisfying the conditions~\eqref{eq:U_generic} when applied to system~\eqref{eq:bl_algo}, (ii) provide a function $W$ satisfying the conditions~\eqref{eq:W_generic} when system~\eqref{eq:rs_algo} is considered, and (iii) the Lipschitz continuity of the dynamics of~\eqref{eq:sp_system} and $\pzeq(\cdot)$ (cf.~\eqref{eq:equilibrium}).
		As for points (i) and (ii), we invoke Lemma~\ref{lemma:bl_algo} and Lemma~\ref{lemma:rs_algo}, respectively, to claim that, for all $\alpha \in (0,1)$, $\pr \in (0,1)$, and $\step \in (0,2\str/( NL^2))$, these points are satisfied.
		Finally, the Lipschitz continuity of the gradients of the objective functions (cf. Assumption~\ref{ass:strong_and_lipschitz}) allows us to claim that point (iii) is satisfied too.
		Therefore, we can apply Theorem~\ref{th:theorem_generic} which allows us to guarantee that there exists $\bar{\pr} > 0$ such that, for all $\pr \in (0,\bar{\pr})$, the point $(0,\pzeq(0))$ is globally exponentially stable for system~\eqref{eq:sp_system}.
		The proof follows by turning out to the original coordinates $(\x,\z)$.

			\section{Asynchronous and Lossy Networks}
			\label{sec:asynchronous_and_lossy}

			In this section, we study the case with imperfect networks.
			Specifically, we consider the case in which the agents are (possibly) asynchronous and communicate with (possible) packet losses.
			More formally, for all $i \in \until{N}$, \GC{we introduce 
   an unknown deterministic sequence $\{\lit \in \{0,1\}\}_{\iter\in\N}$} modeling the fact that agent $i$ is active or not in the %
   sense
			\begin{align*}
				\lit  &= 1  \implies\text{$i$ updates and transmits variables at $\iter$}
				\\
				\lit &= 0 \implies  \text{$i$ does not update and transmit variables at $\iter$}.
			\end{align*}
			Further, for each pair $(i,j) \in \cE$, let \GC{$\{\bijt \in \{0,1\}\}_{\iter\in \N}$ be 
   the unknown deterministic sequence}
   modeling the packet losses in the sense
    \begin{align*}
        \bijt &= 1
         \implies \text{$i$ receives message $\msgjit$ from $j$ at $\iter$}
        \\
        \bijt &= 0
         \implies \text{$i$ does not receive message $\msgjit$ from $j$ at $\iter$},
    \end{align*}
    in which $\msgijt$ has the same meaning as in~\eqref{eq:msgjit}.
    \begin{remark}
        We note that each variable $\bijt$ depends on variable $\ljt$.
        Indeed, the fact that agent $j$ is active at iteration $\iter$ is a necessary (but not sufficient) condition to ensure that agent $i$ successfully receives a message from agent $j$.\oprocend
    \end{remark}
    \GC{We also introduce the variable $\psiijt := \lit\bijt$ for each pair $(i,j) \in \cE$.
    In light of the above definitions, $\psiijt = 1$ means that agent $i$ is active and receives the message $\msgjit$ from agent $j$ at iteration $\iter$.}
   \GC{
   The next assumption enforces an \emph{essentially cyclic} behavior for the variables $\lit$ and $\psiijt$.
   \begin{assumption}[Network Imperfections]\label{ass:imperfections}
				There exists $\Tmax \in \N$ such that, for all $\iter \in \N$, $i \in \until{N}$, and $j \in \cN_i$, there exist $\tau_1,\tau_2 \in [\iter+1,\iter+\Tmax]$ such that $\lambda_i(\tau_1) = \psi_{ij}(\tau_2) = 1$.
                Moreover, for all $i \in \until{N}$, the limit $\lim_{T\to\infty}\frac{1}{T}\sum_{\initer=\iter+1}^T\lita$ exists and is finite uniformly in $\iter$.                
                \oprocend
			\end{assumption}
   }

			\subsection{\ralgo/: Algorithm Design and Convergence Properties}

			To address the challenging framework described above, we propose a slightly different version of Algorithm~\ref{algo:algo} that we call \ralgo/ and report in Algorithm~\ref{algo:ralgo}.
			\begin{algorithm}[H]
				\begin{algorithmic}
					\State \textbf{Initialization}: $\x_i^0 \in \R^{\n}, \z_{ij}^0 \in \R^{2\n} \hspace{.1cm} \forall j \in \cN_i$.
					\For{$\iter=0, 1, \dots$}
					\If{active}
					\vspace{.1cm}
					\State $\begin{bmatrix}
						\yit
						\\
						\sit
					\end{bmatrix} = \frac{1}{1+\rho \dii}\left(\begin{bmatrix}\xit\\
						\nabla f_i(\xit)\end{bmatrix} + \sum_{j\in\cN_i} \zijt\right)$
						\vspace{.1cm}
					\State $\xitp = \xit + \pr\left(\yit - \xit\right)- \pr\step\sit$
						\For{$j \in \cN_i$}
									\State $\msgijt = -\zijt + 2\rho\begin{bmatrix}
										\yit
										\\
										\sit
									\end{bmatrix}$
									\State transmit $\msgijt$ to $j$ and receive $\msgjit$ to $j$
									\If{$\msgjit$ is received}
										\State $\zijtp = (1-\alpha)\zijt + \alpha\msgjit$
									\EndIf
						\EndFor
					\EndIf
					\EndFor
				\end{algorithmic}
				\caption{\ralgo/ (Agent $i$)}
				\label{algo:ralgo}
			\end{algorithm}
			In Algorithm~\ref{algo:ralgo}, we note that each variable $\zijt$ is updated only if, at iteration $\iter$, the message from agent $j$ has been received by agent $i$.
			In other words, the update~\eqref{eq:zij_update} is performed only if $\bijt =1$.
			Further, as one may expect from the above discussion, the additional condition to perform such an update is that agent $i$ is active at iteration $\iter$.
			The next theorem assesses the convergence properties of Algorithm~\ref{algo:ralgo}.
			\begin{theorem}\label{th:convergence_ralgo}
				Consider Algorithm~\ref{algo:ralgo} and let Assumptions~\ref{ass:network},~\ref{ass:strong_and_lipschitz}, and~\ref{ass:imperfections} hold.
				Then, there exist $\bar{\pr}_{\text{R}}, a_3, a_4 > 0$ such that, for all $\rho > 0$, \GC{$\alpha \in (0,1)$}, $\pr \in (0,\bar{\pr}_{\text{R}})$, $\GC{\step \in (0,4\str/( NL^2))}$, $(\x_i^0,\z_i^0) \in \R^\n \times \R^{2\n\degree_i}$, for all $i \in \until{N}$, it holds
				\begin{align*}
					\norm{\xit - \xstar} \leq a_3\exp(-a_4\iter).\eqoprocend
				\end{align*}
			\end{theorem}
			The proof of Theorem~\ref{th:convergence_ralgo} is provided in Section~\ref{sec:proof_ralgo}.
			More in detail, the proof takes on (i) a singular perturbation interpretation to handle the interaction between the variables $\x_i$ and $\z_i$, and (ii) averaging theory to overcome the imperfections of the networks formalized in Assumption~\ref{ass:imperfections}.
			\begin{remark}
				We remark that Assumption~\ref{ass:imperfections} models the scenario of time-varying networks (with or without packet losses) as a special case.
				Indeed, one may interpret $\bijt$ as the \GC{variable modeling the fact} that edge $(i,j)$ exists at iteration $\iter$.
				From this perspective, Assumption~\ref{ass:imperfections}, combined with Assumption~\ref{ass:network}, guarantees that the time-varying graph has some connectivity-like property on time windows of length $\Tmax$.
				Therefore, Theorem~\ref{th:convergence_ralgo} guarantees the linear convergence for \ralgo/ also in the case of time-varying networks. \oprocend
			\end{remark}

			\section{\ralgo/: Stability Analysis}
			\label{sec:analysis_ralgo}

			Here, we provide the analysis of \ralgo/. 
			First, in Section~\ref{sec:ralgo_reformulation}, we interpret the aggregate form of \ralgo/ as a \emph{singularly perturbed} system.
			Then, in Section~\ref{sec:bl_ralgo} and~\ref{sec:rs_ralgo}, respectively, we separately study the identified subsystems by relying on Lyapunov and averaging theory.
			Finally, in Section~\ref{sec:proof_ralgo}, we use the results collected in the previous steps to prove Theorem~\ref{th:convergence_ralgo}.
            \GC{Throughout the whole section, the assumptions of Theorem~\ref{th:convergence_ralgo} hold true.}
			
			\subsection{\ralgo/ as a Singularly Perturbed System}
			\label{sec:ralgo_reformulation}

            We compactly rewrite the local update of \ralgo/ as
			\begin{subequations}\label{eq:ralgo_aggregate}
				\begin{align}
					\xtp &= \xt + \pr \Lt\left(H\left(\xt + \Ax\T\zt\right) -\xt\right)
					\notag\\
					&\hspace{.4cm}
					- \pr \step\Lt H  \left(\G(\xt) + \An\T\zt\right)\label{eq:ralgo_aggregate_x}
					\\
					\ztp &= \zt \! - \! \alpha\Psit\left((I \! + \! P)\zt \! - \! 2\rho PA\cH (A\T\zt \! + \! \vv(\xt))\right)\!,\label{eq:ralgo_aggregate_z}
				\end{align}
			\end{subequations}
			where $\xt$, $\zt$, $H$, $\Ax$, $\An$, $P$, and $\vv$ have the same meaning as in~\eqref{eq:algo_aggregate_form} and we further introduced the matrices $\Lt := \blkdiag(\lambda_1(\iter) I_\n,\dots,\lambda_N(\iter) I_{\n}) \in \R^{N\n \times N\n}$ and $\Psit := \blkdiag(\Psi_1(\iter)I_{2\n d_1},\dots,\Psi_N(\iter)I_{2\n d_N}) \in \R^{2\n\dd}$ in which $\Psiit := \blkdiag\left(\col(\psiijt)_{j\in \cN_i} \otimes I_{2\n}\right)\in \R^{2\n\dii \times 2\n\dii}$ for all $i \in \until{N}$.
			Fig.~\ref{fig:block_diagram_ralgo} reports a block diagram graphically describing~\eqref{eq:ralgo_aggregate}.
			\begin{figure}[H]
				\includegraphics[width=\columnwidth]{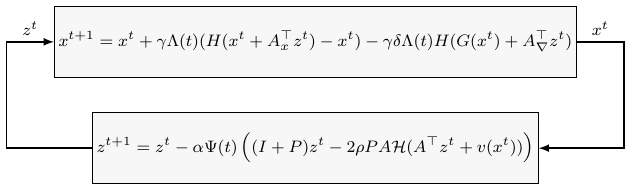}
				\caption{Block diagram representing~\eqref{eq:ralgo_aggregate}.}
				\label{fig:block_diagram_ralgo}
			\end{figure}
			By using the change of variables~\eqref{eq:z_perp_z_bar} and Lemma~\ref{lemma:results}, we rewrite~\eqref{eq:ralgo_aggregate} as 
			\begin{subequations}\label{eq:ralgo_transformed_with_bz}
				\begin{align}
					\xtp &= \xt + \pr \Lt\left(H\left(\xt + \Ax\T\M\T\pzt\right) -\xt\right)
					\notag\\
					&\hspace{.4cm}
					- \pr\step\Lt H  \left(\G(\xt) + \An\T\M\T\pzt\right)\label{eq:ralgo_transformed_with_bz_x}
					\\
					\bztp &= \bzt - \alpha\B\T\Psit\left(I + P - 2\rho PA\cH A\T\right)\M\pzt
					\notag
					\\
					&\hspace{.4cm}
					+ \alpha 2\rho\B\T\Psit PA\cH \vv(\xt)					
					\label{eq:ralgo_transformed_with_bz_bz}
					\\
					\pztp &= \pzt - \alpha\M\T\Psit\left(I + P - 2\rho PA\cH A\T\right)\M\pzt
					\notag
					\\
					&\hspace{.4cm}
					+ \alpha 2\rho\M\T\Psit PA\cH \vv(\xt).
					\label{eq:ralgo_transformed_with_bz_pz}
				\end{align}
			\end{subequations}
			As in Section~\ref{sec:analysis_algo}, we observe that $\bz$ does not affect the other states of system~\eqref{eq:ralgo_transformed_with_bz} and, thus, we will ignore it in our analysis.
			Moreover, by using $\txt = \xt - \oneNn\xstar$ to transform $\xt$, system~\eqref{eq:ralgo_transformed_with_bz} is equivalent to
			\begin{subequations}\label{eq:ralgo_transformed}
				\begin{align}
						\txtp &= \txt + \pr \Lt\left(H\left(\xt + \Ax\T\M\T\pzt\right) -\xt\right)
					\notag\\
					&\hspace{.4cm}
					- \pr\step\Lt H  \left(\G(\xt) + \An\T\M\T\pzt\right)\label{eq:ralgo_transformed_x}
					\\
					\pztp &= \pzt - \alpha\M\T\Psit\left(I + P - 2\rho PA\cH A\T\right)\M\pzt
					\notag
					\\
					&\hspace{.4cm}
					+ \alpha 2\rho\M\T\Psit PA\cH \vv(\txt + \oneNn\xstar),
					\label{eq:ralgo_transformed_pz}
				\end{align}
			\end{subequations}
			where, with a slight abuse of notation, we maintained a hybrid notation with $\xt = \oneNn\mut + \rr\pxt + \oneNn\xstar$.
			System~\eqref{eq:ralgo_transformed} is a singularly perturbed system.
			Specifically,~\eqref{eq:ralgo_transformed_x} plays the role of the slow subsystem, while~\eqref{eq:ralgo_transformed_pz} plays the role of the fast one.
			As in system~\eqref{eq:sp_system}, the subsystem~\eqref{eq:ralgo_transformed_pz} has an equilibrium $\pzeq(\txt)$ parametrized in the slow state $\txt$. 
			Indeed, by using $\x = \tx + \oneNn\xstar$, one may check this claim through the following chain of equations
			\begin{align}
				&(I + P - 2\rho PA\cH A\T)\M\pzeq(\txt) - 2\rho PA\cH \vv(\x)
				\notag\\
				&\stackrel{(a)}{=}
				2\rho\left((I + P - 2\rho PA\cH A\T)(\M\F\inv M\T) - I\right)PA\cH \vv(\x)
				\notag\\
				&\stackrel{(b)}{=}
				2\rho\left(\M\M\T - I\right)PA\cH  \vv(\x)
				\notag\\
				&\stackrel{(c)}{=}
				-2\rho\B\B\T PA\cH  \vv(\x)= 0,\label{eq:equilibrium_ralgo}
			\end{align}
			where in $(a)$ we used the definition of $\pzeq$ (cf.~\eqref{eq:equilibrium}), in $(b)$ we used the result $(I + P - 2\rho PA\cH A\T)\M\F\inv M\T = \M\M\T$, in $(c)$ we used $\M\M\T = I - \B\B\T$, while the last equality follows by~\eqref{eq:B_P_A}.

			\subsection{\GC{Boundary Layer System}}
			\label{sec:bl_ralgo}

			Here, we study the boundary layer system of~\eqref{eq:ralgo_transformed}, i.e., following the same steps of Section~\ref{sec:bl_algo}, we study the subsystem~\eqref{eq:ralgo_transformed_pz} by considering an arbitrarily fixed $\txt = \tx \in \R^{N\n}$ and the error variable $\tpzt = \pzt - \pzeq(\tx)$.
			Therefore, using the result in~\eqref{eq:equilibrium_ralgo}, we get
			\begin{align}\label{eq:bl_ralgo}
				\tpztp = \tpzt - \alpha\hFt\tpzt,
			\end{align}
            \GC{
            where $\hF: \N \to \R^{\m \times \m}$ reads as 
            \begin{align*}
                \hFt := \M\T\Psit\left(I + P - 2\rho PA\cH A\T\right)\M.
            \end{align*}
			The next lemma ensures that the origin is globally exponentially stable for~\eqref{eq:bl_ralgo}.
			\begin{lemma}\label{lemma:bl_ralgo}
				Consider~\eqref{eq:bl_ralgo}. 
				Then, the origin is a globally exponentially stable equilibrium point for~\eqref{eq:bl_ralgo} for all $\alpha \in (0,1)$.
			\end{lemma}
			\begin{proof}
                We will proceed by studying the spectrum of the matrix $\prod_{\initer=\iter+1}^{\iter+\Tmax}\hFta$ for all $\iter \in \N$.
                To this end, let $T \in \R^{2\n\dd \times 2\n\dd}$ and $\hT: \N \to \R^{2\n\dd \times 2\n\dd}$ be defined as 
                \begin{subequations}
                    \begin{align}
                        T &:= (1-\alpha)I + \alpha(-P + 2\rho PA\cH A\T)
                        \\
                        \hTt &:= I + \Psit(T - I) 
                        = I -\Psit + \Psit T.\label{eq:hTt}
    				\end{align}
                \end{subequations}
				To simplify the computations, let us start with the case in which $T$ is diagonalizable.
                Then, we decompose each $\z \in \R^{2\n\dd}$ as a linear combination of the eigenvectors $v_i$ of $T$, namely 
				\begin{align}
					T\z 
                    = T\sum_{i=1}^{2\n\dd}a_i v_i
					&
                    \stackrel{(a)}{=}
					\sum_{i=1}^{2\n\dd}a_i\nu_iv_i
                    \stackrel{(b)}{=} 
					\sum_{i=1}^{\bb}a_i v_i + \sum_{i=\bb+1}^{2\n\dd}a_i\nu_iv_i\label{eq:eigenvectors}
					,
				\end{align}
				where in $(a)$, for all $i \in \until{2\n\dd}$, we use the eigenvalue $\nu_i$ associated to the eigenvector $v_i$, while in $(b)$ we use the fact that $\bb$ eigenvalues are equal to $1$~\cite{bastianello2020asynchronous} and, without loss of generality, we pose them as the first ones.
				Let us decompose $v_i$ as $v_i := \col(v_{1,i},\dots,v_{1,2\n\dd})$ for all $i \in \until{2\n\dd}$ and
				let $n_{k}(\iter) \in \N$ be the number of iterations in which the $k$-th component of $\Psitau$ is $1$ for $\tau \in \{\iter+1,\iter+\Tmax\}$ and $k \in \until{2\n\dd}$.
                Then, we focus on the time interval $\{\iter+1,\dots,\iter+\Tmax\}$ and, by combining this notation with~\eqref{eq:hTt} and~\eqref{eq:eigenvectors}, we get
				\begin{align}
					\prod_{\initer=\iter+1}^{\iter+\Tmax}\hTta\z 
					= \sum_{i=1}^{\bb}
					\begin{bmatrix}
						a_i v_{1,i}
						\\
						\vdots 
						\\
						a_i v_{2\n\dd,i}
					\end{bmatrix} 
					+ \sum_{i=\bb+1}^{2\n\dd}
					\begin{bmatrix}
						\nu^{n_1(\iter)}_i  a_i v_{1,i}
						\\
						\vdots
						\\
						\nu^{n_{2\n\dd}(\iter)}_i a_i v_{2\n\dd,i}
					\end{bmatrix}.\label{eq:prod_htT}
				\end{align}
				In light of Assumption~\ref{ass:imperfections}, we note that $n_{k}(\iter) > 1$ for all $k \in \until{2\n\dd}$ and $\iter\in\N$.
				Moreover, we recall that the eigenvalues $\nu_i$ with $i \in \{\bb+1,\dots,2\n\dd\}$ lie in the open unit disc~\cite{bastianello2020asynchronous}, i.e., $|\nu_i| < 1$ for all $i \in \{\bb+1,\dots,2\n\dd\}$.
				Moreover, since the unitary eigenvalues of $T$ are semi-simple~\cite{bastianello2020asynchronous}, by leveraging the concept of generalized eigenvectors and the above arguments, we can easily recover a result analogous to~\eqref{eq:prod_htT} (with a finite correcting factor $\kappa$) also in the case in which $T$ is not diagonalizable.
				In both cases, by construction, the matrix $\M$ is orthogonal to the unitary eigenspace of $T$ and, thus, for all $\iter_0, \iter \in \N$ with $\iter \ge \iter_0$, the result~\eqref{eq:prod_htT} leads to 
				\begin{align}
					\norm{\prod_{\tau=\iter_0}^{\iter}\left(I - \alpha\hFta\right)} &= \norm{\prod_{\tau=\iter_0}^{\iter}\M\T\hTta\M} 
					\notag\\
					&\leq \frac{\kappa}{|\nu_i|^{\Tmax}}\nu_{\text{max}}^{\iter - \iter_0},\label{eq:norm_hTt}
				\end{align}
				where $\nu_{\text{max}} := \left(\max_{i \in \{\bb+1,\dots,2\n\dd\}}|\nu_i|\right)^\frac{\Tmax - 1}{\Tmax}$ and
				\begin{align*}
					\begin{cases}
						\kappa  = 1 \text{ if $T$ is diagonalizable} 
						 \\
						 \kappa \ge 1 \text{ if $T$ is not diagonalizable}.
					\end{cases}
				\end{align*}
                Since $\nu_{\text{max}} < 1$, the property~\eqref{eq:norm_hTt} allows us to invoke the result~\cite[Lemma~1]{zhou2017asymptotic} about linear time-varying discrete-time systems and conclude the proof.
 			\end{proof}
            }

			\subsection{Reduced System: Averaging Analysis}
			\label{sec:rs_ralgo}

			In this section, we will study the reduced system associated to~\eqref{eq:ralgo_transformed}, i.e., the system obtained by plugging $\pzt = \pzeq(\txt)$  into~\eqref{eq:ralgo_transformed_x} for all $\iter\in\N$.
			Then, by enforcing the definition of $\pzeq$ (cf.~\eqref{eq:equilibrium}), the results in~\eqref{eq:results_equilibrium}, $\oneNn \in \ker\left(\frac{\oneNn\oneNn\T}{N} - I\right)$, and $\oneNn\T\rr = 0$, such a reduced system reads as
			\begin{align}\label{eq:rs_ralgo}
				\txtp &= \txt + \pr \Lt\left(I - \frac{\oneNn\oneNn\T}{N}\right)\txt 
				\notag\\
				&\hspace{.4cm}
				- \pr \step\Lt\frac{\oneNn\oneNn\T}{N}\G(\txt + \oneNn\xstar).
			\end{align}
			To simplify the notation, we equivalently write system~\eqref{eq:rs_ralgo} as
			\begin{align}\label{eq:rs_ralgo_r}
				\txtp &= \txt + \pr r(\txt,\iter),
			\end{align}
			where we introduced $r: \R^{N\n} \times N \to \R^{N\n}$ defined as 
			\begin{align}
				r(\tx,\iter) &= \Lt\left(\frac{\oneNn\oneNn\T}{N} - I\right)\tx
				\notag\\
				&\hspace{.4cm}
				- \step \Lt\frac{\oneNn\oneNn\T}{N}\G(\tx + \oneNn\xstar).
			\end{align}
			\GC{We note that $r(0,\iter) = 0$ for all $\iter \in \N$.}
			Since system~\eqref{eq:rs_ralgo} is time-varying, we study it by resorting to averaging theory.
			To this end, let the function $\ra: \R^{N\n} \to \R^{N\n}$ be defined as 
			\begin{align*}
				&\ra(\tx) 
				:= \lim_{T \to \infty}\frac{1}{T}\sum_{\initer = \bar{\iter} + 1}^{\bar{\iter}+ T} r(\tx,\initer).
			\end{align*}
			\GC{In light of Assumption~\ref{ass:imperfections},}
   such a limit exists uniformly in $\bar{\iter}$ and for all $\tx \in \R^{N\n}$.
			Specifically, it reads as
			\begin{align}
				\ra(\tx) &= \EL\left(\frac{\oneNn\oneNn\T}{N} - I\right)\tx
				\notag\\
				&\hspace{.4cm}
				-\step\EL\frac{\oneNn\oneNn\T}{N}\G(\tx + \oneNn\xstar),\label{eq:ra}
			\end{align}
			\GC{where we introduced $\EL := \blkdiag(\lambda_{1,\text{av}}I,\dots,\lambda_{N,\text{av}}I)$, with $\Eli :=\lim_{T\to\infty}\frac{1}{T}\sum_{\initer=\iter+1}^T\lita$ for all $i \in \until{N}$.}
			The next lemma provides a suitable function which turns out to be crucial in concluding the proof of Theorem~\ref{th:convergence_ralgo}. %
			\begin{lemma}\label{lemma:rs_ralgo}
				Consider $\ra(\cdot)$ as defined in~\eqref{eq:ra}.
				Then, there exist $ \beta_1, \beta_2, \beta_3, \beta_4 > 0$, and $\Wa: \R^{N\n} \to \R$ such that, for all $\GC{\step \in (0,4\str/( NL^2))}$ and $\txa \in \R^{N\n}$, it holds 
				\begin{subequations}\label{eq:r_conditions}
					\begin{align}
						\beta_1 \norm{\txa}^2 \leq \Wa(\txa) &\leq \beta_2\norm{\txa}^2\label{eq:r_conditions_first_bound_generic}
						\\
						\nabla\Wa(\txa)\T \ra(\txa) &\leq - \beta_3\norm{\txa}^2\label{eq:r_conditions_minus_generic}
						\\
						\norm{\nabla\Wa(\txa)}&\leq \beta_4\norm{\txa}.\label{eq:r_conditions_bound_generic}
					\end{align}
				\end{subequations}
			\end{lemma}
			\begin{proof}
				Let us rewrite~\eqref{eq:ra} in a block-wise sense.
				To this end, for all $i \in \until{N}$, let $\rai: \R^{N\n} \to \R^{\n}$ be defined as
				\begin{align}
					\rai(\txa) &=  \Eli \left(\frac{1}{N}\sum_{j=1}^N \txaj - \txai\right) 
					\notag\\
					&\hspace{.4cm}
					-  \step\Eli\oneNn\T\G(\txt+\oneNn\xstar)/N,\label{eq:r_i_av}
				\end{align}
				where we decomposed $\txa$ according to $\txa := \col(\tx_{1,\text{av}}, \dots, \tx_{N,\text{av}})$ with $\txai \in \R^\n$ for all $i \in \until{N}$.
				Then, it holds $\ra(\tx) := \col(r_{1,\text{av}}(\txa),\dots,r_{N,\text{av}}(\txa))$.
				Now, let us consider the function $\Wai: \R^{\n} \to \R$ defined as 
				\begin{align*}
					\Wai(\txai) = \frac{1}{2\Eli}\norm{\txai}^2.
				\end{align*}
				Then, \GC{it holds $\nabla\Wai(\txai) = \frac{\txai}{\Eli}$ which, combined with~\eqref{eq:r_i_av}, leads to}
				\begin{align}
					&
                    \nabla \Wai(\txai)\T \rai(\txa) 
                    \notag\\
                    &= \! - \! \norm{\txai}^2 \! + \! \txai\T \frac{1}{N}\sum_{j=1}^N \txaj 
					\! - \! \step\txai\T \frac{\oneNn\T}{N}\G(\txa+\oneNn\xstar)
					\notag\\
					&\stackrel{(a)}{=}
					\! - \! \norm{\txai}^2 \! + \! \txai\T \mua
					\! - \! \step\txai\T \frac{\oneNn\T}{N}\G(\txa\GC{+\oneNn\xstar}),\label{eq:nabla_Wai}
				\end{align}
				\GC{where in $(a)$ we used $\mua := \sum_{j=1}^N \txaj/N$.}
				Let $\Wa(\txa) := \sum_{i=1}^N\Wai(\txai)$.
				\GC{In light of the essentially cyclic property enforced in Assumption~\ref{ass:imperfections}, there exists $\el \in (0,1)$ such that, for all $i \in \until{N}$, $\Wa$ is positive definite and, thus, the conditions~\eqref{eq:r_conditions_first_bound_generic} and~\eqref{eq:r_conditions_bound_generic} are satisfied.}
				In order to check~\eqref{eq:r_conditions_minus_generic}, we use~\eqref{eq:nabla_Wai} obtaining
				\begin{align}
					&
                    \nabla \Wa(\txa)\T \GC{\ra(\txa)}
                    \notag\\
                    &= - \norm{\txa}^2 + N \mua\T \mua - \step\mua\T\oneNn\T\G(\txa\GC{+\oneNn\xstar})
					\notag\\
					&\stackrel{(a)}{=} - \norm{\oneNn\mua + \rr\pxa}^2 + N \mua\T \mua 
					\notag\\
					&\hspace{.4cm}
					- \step\mua\T\oneNn\T\G(\txa\GC{+\oneNn\xstar})
					\notag\\
					&\stackrel{(b)}{=} - \norm{\pxa}^2 - \step\mua\T\oneNn\T\G(\txa\GC{+\oneNn\xstar})
					\notag\\
					&\stackrel{(c)}{=} - \norm{\pxa}^2 - \step\mua\T\nabla f(\mua + \xstar) 
					\notag\\
					&\hspace{.4cm}
					+ \step\mua\T N\tilde{\G}(\muat,\pxa),\label{eq:nablaW} 
				\end{align}
				where in $(a)$ we used $\pxa := \rr\T\txa$ and decomposed $\txa$ according to $\txa = \oneNn\mua + \rr\pxa$, in $(b)$ we expanded the square norm and used $\rr\T\rr = I$, $\oneNn\T \oneNn = N I$, and $\rr\T \oneNn = 0$, while in $(c)$ we added and subtracted the term $\step\mua\T\nabla f(\mua + \xstar)$ and used the operator $\tilde{\G}(\cdot,\cdot)$ with the same meaning as in~\eqref{eq:deltaG}.
				By using the strong convexity of $f$ (cf. Assumption~\ref{ass:strong_and_lipschitz}) and $\nabla f(\xstar) = 0$, we bound~\eqref{eq:nablaW} as 
				\begin{align}
					\Wa(\txa)\T \txa &\leq \! - \! \norm{\pxa}^2 \! - \! \step\str\norm{\mua}^2
					\! + \! \step\mua\T N\tilde{\G}(\muat,\pxa)
					\notag\\
					&\stackrel{(a)}{\leq}
					\! - \! \norm{\pxa}^2 
					\! - \! \step\str\norm{\mua}^2
					\! + \! \step\lip\sqrt{N}\norm{\mua}\!\norm{\pxa}
					\notag\\
					&\stackrel{(b)}{\leq}
					-\begin{bmatrix}
						\norm{\mua}
						\\
						\norm{\pxa} 
					\end{bmatrix}\T Q\av(\step)\begin{bmatrix}
						\norm{\mua}
						\\
						\norm{\pxa}
					\end{bmatrix}
					,\label{eq:nablaW_2} 
				\end{align}
				where in $(a)$ we used the Cauchy-Schwarz inequality and the result~\eqref{eq:bound_delta_G} to bound the term $\mua\T N\tilde{\G}(\muat,\pxa)$, while in $(b)$ we rearranged the inequality in a matrix form by introducing $Q\av(\step) = Q\av(\step)\T \in \R^{2 \times 2}$ defined as 
				\begin{align*}
					Q\av(\step) = \begin{bmatrix}
						\step\str& -\step\lip\sqrt{N}/2
						\\
						-\step\lip\sqrt{N}/2& 1
					\end{bmatrix}.
				\end{align*}
				By Sylvester Criterion, it holds $Q\av(\step) > 0$ if and only if $\step < \frac{4\str}{\lip^2 N}$.
				Therefore, by choosing any $\step \in (0,\frac{4\str}{\lip^2 N})$, the inequality~\eqref{eq:nablaW_2} leads to 
				\begin{align*}
					\nabla\Wa(\txa)\T \ra(\txa) &\leq - q\av\left(\norm{\mua}^2 + \norm{\pxa}^2\right)
					\notag\\
					&= -q\av \txa\T \left(\frac{\oneNn\oneNn\T}{N^2} + \rr\rr\T\right)\txa,
				\end{align*}
				where $q\av > 0$ is the smallest eigenvalue of the positive definite matrix $Q\av(\step)$.
				Thus, since $\frac{\oneNn\oneNn\T}{N^2} + \rr\rr\T > 0$, also~\eqref{eq:r_conditions_minus_generic} is ensured and the proof is concluded.
			\end{proof}

			\subsection{Proof of Theorem~\ref{th:convergence_ralgo}}
			\label{sec:proof_ralgo}

			The proof of Theorem~\ref{th:convergence_ralgo} is based on a proper exploitation of Theorem~\ref{th:theorem_generic}.
			\GC{More in details, to apply Theorem~\ref{th:convergence_ralgo}, we first need to use a result arising in averaging theory, i.e.,~\cite[Prop.~7.3]{bof2018lyapunov}.}
			Indeed, in order to apply Theorem~\ref{th:theorem_generic} we need to (i) provide a function $U$ satisfying the conditions~\eqref{eq:U_generic} when applied to system~\eqref{eq:bl_ralgo}, (ii) provide a function $W$ satisfying the conditions~\eqref{eq:W_generic} when system~\eqref{eq:rs_ralgo} is considered, and (iii) guarantee the Lipschitz continuity of the dynamics of~\eqref{eq:ralgo_transformed} and $\pzeq(\cdot)$ (cf.~\eqref{eq:equilibrium}).
			\paragraph{Point (i)} we recall that Lemma~\ref{lemma:bl_ralgo} guarantees that the origin is a globally exponentially stable equilibrium point of system~\eqref{eq:bl_ralgo}.
			Thus, the function $U$ satisfying~\eqref{eq:U_generic} can be obtained by applying the Converse Lyapunov Theorem~\cite[Th.~5.8]{bof2018lyapunov}.
			Indeed, we note that system~\eqref{eq:bl_ralgo} is independent with respect to the slow state $\txt$.
			Thus, the uniformity with respect to $\txt$ of the conditions~\eqref{eq:U_generic} is guaranteed.
			\paragraph{Point (ii)} by invoking Lemma~\ref{lemma:rs_ralgo} and the Lipschitz continuity of the gradients of the objective functions (cf. Assumption~\ref{ass:strong_and_lipschitz}), we can apply~\cite[Prop.~7.3]{bof2018lyapunov}
            which guarantees the existence of a function $W$ and $\bstepa > 0$ such that, for all $\pr \in (0,\bstepa)$, $W$ satisfies the conditions~\eqref{eq:W_generic} along the trajectories of the reduced system~\eqref{eq:rs_ralgo}.
			\paragraph{Point (iii)} Finally, the Lipschitz continuity of the gradients of the objective functions (cf. Assumption~\ref{ass:strong_and_lipschitz}) allows us to claim that point (iii) is satisfied too.
			
			Once these three points have been checked, we apply Theorem~\ref{th:theorem_generic} and, thus, guarantee that there exists $\bar{\pr}_{\text{R}} \in (0,\bstepa)$ such that, for all $\pr \in (0,\bar{\pr}_{\text{R}})$, the point $(0,\pzeq(0))$ is globally exponentially stable for system~\eqref{eq:ralgo_transformed}.
			The proof concludes by turning out to the original coordinates $(\x,\z)$. 

			\section{Robustness against inexact computations and communications}
			\label{sec:robustness}

			In this section, we deal with possible errors in the algorithm updates (due to, e.g., quantization effects, disturbances affecting the inter-agent communication, or stochastic gradients).
			In this framework, for all $i \in \until{N}$, we introduce $\errxit \in \R^\n$ and $\errzit \in \R^{2\n\dd_i}$ to model the errors affecting the update of $\xit$ and $\zit$, respectively, at time $\iter$.
			Then, by considering also asynchronous updates and packet losses in the communication (cf. Section~\ref{sec:asynchronous_and_lossy}), we study the disturbed version of~\eqref{eq:ralgo_aggregate}, namely%
				\begin{subequations}\label{eq:algo_with_errrors_aggregate}
					\begin{align}
						\xtp &= \xt + \pr \Lt\left(H\left(\xt + \Ax\T\zt\right) - \xt\right)
						\notag\\
						&\hspace{.4cm}
						- \pr \step\Lt H  \left(\G(\xt) + \An\T\zt\right) + \errxt\label{eq:algo_with_errrors_x_aggregate}
						\\
						\ztp &= \zt \! - \! \alpha\Psit\left((I \! + \! P)\zt \! - \! 2\rho PA\cH (\zt \! + \! \vv(\xt))\right)\! + \errzt,\label{eq:algo_with_errrors_z_aggregate}
					\end{align}
				\end{subequations}
                where $\xt$, $\zt$, $\G(\xt)$, and all the matrices have the same meaning as in~\eqref{eq:algo_aggregate_form} and~\eqref{eq:ralgo_aggregate}, \GC{while $\errxt$ and $\errzt$ read as}
                \begin{align}
				\errxt := \begin{bmatrix}\err_{1,\x}^\iter\\ 
					\vdots
					\\
					\err_{N,\x}^\iter\end{bmatrix}, \quad \errz := \begin{bmatrix}\err_{1,\z}^\iter
						\\
						\vdots
						\\
						\err_{N,\z}^\iter\end{bmatrix}.\label{eq:rx_rz}
			\end{align}
			The next result characterizes the robustness of~\eqref{eq:algo_with_errrors_aggregate} with respect to $\err := \col(\errx,\errz) \in \R^{(N + 2\dd)\n}$
   in terms of ISS property.
			\begin{theorem}\label{th:error}
				Consider~\eqref{eq:algo_with_errrors_aggregate} and let Assumptions~\ref{ass:network},~\ref{ass:strong_and_lipschitz}, and~\ref{ass:imperfections} hold.
				Then, there exist a $\mathcal{KL}$ function $\beta(\cdot)$, a $\mathcal{K}_\infty$ function $\nu(\cdot)$, and some constants $\bar{\pr}_{\text{R}}, \bar{\alpha}_{\text{R}}, \bpr, c^0 > 0$ such that, for all $\rho >0$, $\alpha \in (0,\bar{\alpha}_{R})$, $\pr \in (0,\bar{\pr}_{\text{R}})$, $\step \in (0,\bpr)$, $(\x^0,\z^0) \in \R^{N\n} \times \R^{2\n\dd}$, it holds
				\begin{align*}
					\norm{\xit- \xstar}\leq \beta(c^0,\iter) + \nu(\norm{\err}_\infty),
				\end{align*}
				for all $i \in \until{N}$, $\err \in \mathcal{L}_\infty^{(N + \dd)\n}$, and $\iter \in \N$.\footnote{%
				See~\cite[Chapter 4]{khalil2002nonlinear} for the function classes' definitions.}
		\end{theorem}
		\begin{proof}
				As in Section~\ref{sec:analysis_algo} and~\ref{sec:analysis_ralgo}, we rewrite system~\eqref{eq:algo_with_errrors_aggregate} by (i) using the changes of coordinates~\eqref{eq:z_perp_z_bar} and $\txt := \xt - \oneNn\xstar$, and (ii) ignoring the evolution of $\bzt$.
				Thus, we get 
				\begin{subequations}
    \label{eq:algo_with_errors_transformed}
					\begin{align}
						\txtp &= \txt + \pr \Lt\left(H\left(\xt + \Ax\T\M\T\pzt\right) -\xt\right)
					\notag\\
					&\hspace{.4cm}
					- \pr\step\Lt H  \left(\G(\xt) + \An\T\M\T\pzt\right) + \errxt %
						\\
						\pztp &= \pzt - \alpha\M\T\Psit\left(I + P - 2\rho PA\cH A\T\right)\M\pzt
						\notag
						\\
						&\hspace{.4cm}
						+ \alpha 2\rho\M\T\Psit PA\cH \vv(\xt) + \M\T\errzt,
					\end{align}
				\end{subequations}
				where, with a slight abuse of notation, we maintained a hybrid notation with $\xt = \txt + \oneNn\xstar$.
				We note that system~\eqref{eq:algo_with_errors_transformed} can be viewed as a perturbed version of~\eqref{eq:ralgo_transformed}.
			To take advantage of this interpretation, let $\xi \in \R^{(N + 2\dd)\n}$ be defined as 
			\begin{align}\label{eq:xi}
				\xi := \begin{bmatrix}
					\tx
					\\
					\pz - \pzeq(\tx)
				\end{bmatrix},
			\end{align} 
			which allows us to compactly rewrite system~\eqref{eq:algo_with_errors_transformed} as 
			\begin{align}\label{eq:xi_system}
				\xi^{\iterp} = \fxi(\xi^\iter,\iter) + \trt,
			\end{align}
			where $\fxi := \R^{(N + 2\dd)\n} \to \R^{(N + 2\dd)\n}$ is suitably defined to contain all the terms arising in~\eqref{eq:algo_with_errors_transformed}, while $\trt \in \R^{(N + 2\dd)\n}$ is defined as $\trt : =\col(\errxt,\M\T \errzt)$.
			In the proof of Theorem~\ref{th:convergence_ralgo} (cf. Section~\ref{sec:proof_ralgo}), by using Theorem~\ref{th:theorem_generic}, we have guaranteed the existence of $\bar{\alpha}_{\text{R}}, \bar{\step_{\text{R}}} > 0$ such that, for all $\alpha \in (0,\bar{\alpha}_{\text{R}})$ and $\pr \in (0,\bar{\pr}_{\text{R}})$, the origin is a globally exponentially stable equilibrium for the nominal system described by 
			\begin{align}\label{eq:xi_system_nominal}
				\xi^{\iterp} = \fxi(\xi^\iter,\iter).
			\end{align}
			Therefore, by using the Converse Lyapunov Theorem~\cite[Th.~5.8]{bof2018lyapunov}, there exist some constants $a_{5}, a_6, a_7, a_8 > 0$ and a Lyapunov function $\Vxi : \R^{(N + 2\dd)\n} \times \N \to \R$ such that 
			\begin{subequations}\label{eq:V_proof_errors}
				\begin{align}
					a_5\norm{\xi}^2 \leq \Vxi(\xi,\iter) &\leq a_6\norm{\xi}^2\label{eq:V_proof_errors_1}
					\\
					\Vxi(\fxi(\xi,\iter),\iterp) - \Vxi(\xi,\iter) &\leq - a_7\norm{\xi}^2\label{eq:V_proof_errors_2}
					\\
					|\Vxi(\xi_1,\iter)  - \Vxi(\xi_2,\iter) &\leq a_8\norm{\xi_1 - \xi_2}(\norm{\xi_1} + \norm{\xi_2}),\label{eq:V_proof_errors_3}
				\end{align}
			\end{subequations}
			for all $\xi, \xi_1, \xi_2 \in \R^{(N + 2\dd)\n}$ and $\iter \in \N$.
			Therefore, by evaluating $\Delta \Vxi(\xi^\iter,\iter) := \Vxi(\xi^\iterp,\iterp) - \Vxi(\xi^\iter,\iter)$ along the trajectories of system~\eqref{eq:xi_system}, we get
			\begin{align}
				\Delta \Vxi(\xi^\iter,\iter)  &= \Vxi(\fxi(\xi^\iter,\iter) + \trt,\iterp) - \Vxi(\xi^\iter,\iter)
				\notag\\
				&\stackrel{(a)}{=}
				\Vxi(\fxi(\xi^\iter,\iter),\iterp) - \Vxi(\xi^\iter,\iter)
				\notag\\
				&\hspace{.4cm}+V(\fxi(\xi^\iter,\iter) + \trt,\iterp) - \Vxi(\fxi(\xi^\iter,\iter),\iterp)
				\notag
				\\
				&\stackrel{(b)}{\leq}
				- a_7\norm{\xi^\iter}^2
				\notag\\
				&\hspace{.4cm}+V(\fxi(\xi^\iter,\iter) + \trt,\iterp) - \Vxi(\fxi(\xi^\iter,\iter),\iterp)
				\notag
				\\
				&\stackrel{(c)}{\leq}
				- a_7\norm{\xi^\iter}^2
				\notag\\
				&\hspace{.4cm}+ \! a_8\norm{\trt}\!(\norm{\fxi(\xi^\iter,\iter) + \trt} \! + \! \norm{\fxi(\xi^\iter,\iter)}),\label{eq:DeltaV_proof_with_errors}
			\end{align}
			where in $(a)$ we add and subtract the term $\Vxi(\fxi(\xi^\iter,\iter),\iterp)$, in $(b)$ we apply~\eqref{eq:V_proof_errors_2}, in $(c)$ we apply~\eqref{eq:V_proof_errors_3}.
			Being the gradients of $f_i$ Lipschitz continuous (cf. Assumption~\ref{ass:strong_and_lipschitz}), then $\fxi(\cdot,\iter)$ is Lipschitz, i.e., there exists $\Lxi > 0$ such that 
			\begin{align}
				\norm{\fxi(\xi_1,\iter) - \fxi(\xi_2,\iter)} \leq \Lxi\norm{\xi_1 - \xi_2},\label{eq:Lxi}
			\end{align}
			for all $\xi_1, \xi_2 \in \R^{N\n + \m}$ and $\iter \in \N$.
			Thus, since the origin is an equilibrium of the nominal system~\eqref{eq:xi_system_nominal}, we can use (i) $\fxi(\xi^\iter,\iter) = 0$ for all $\iter \in N$, (ii) the result~\eqref{eq:Lxi}, and (iii) the triangle inequality to bound the right-hand side of~\eqref{eq:DeltaV_proof_with_errors} as 
			\begin{align}
				\Delta \Vxi(\xi^\iter,\iter) &\leq - a_7\norm{\xi^\iter}^2 + a_8\norm{\trt}^2 + 2a_8\Lxi\norm{\trt}\norm{\xi^\iter}
				\notag\\
				&\stackrel{(a)}{\leq}- (a_7 + \nu a_8\Lxi)\norm{\xi^\iter}^2 +  \left(a_8 + \tfrac{a_8\Lxi}{\nu}\right)\norm{\trt}^2
				,\label{eq:DeltaV_proof_with_errors_2}
			\end{align} 
			where in $(a)$ we use the Young inequality with parameter $\nu > 0$ to bound the term $2a_8\Lxi\norm{\trt}\norm{\xi^\iter}$.
			%
			Hence, by arbitrarily choosing any $a_7^\prime \in (0,a_7)$, if we set $\nu \in (0,(a_7 - a_7^\prime)/(a_8\Lxi))$, the right-hand side of~\eqref{eq:DeltaV_proof_with_errors_2} can be bounded as 
			\begin{align}
				\Delta \Vxi(\xi^\iter,\iter) &\leq - a_7^\prime\norm{\xi^\iter}^2 +  a_{9}\norm{\trt}^2,\label{eq:DeltaV_proof_with_errors_3}
			\end{align}
			with $a_{9} := a_8 + \frac{a_8\Lxi}{\nu}$.
			The result~\eqref{eq:DeltaV_proof_with_errors_3}, combined with~\eqref{eq:V_proof_errors_1}, allows us to claim that $\Vxi$ is an ISS-Lyapunov function (cf.~\cite[Def.~2]{li2018input}) for system~\eqref{eq:xi_system} and, thus, that system~\eqref{eq:xi_system} is ISS~\cite[Th.~1]{li2018input}.
			The proof follows by (i) using the definition of ISS systems~\cite[Def.~1]{li2018input}, (ii) using the inequalities $\norm{\xit - \xstar} \leq \xi^\iter$ and $\norm{\trt} \leq \norm{\errt}$, and (iii) setting $c^0  = \xi^0$.
			\end{proof}
			\begin{remark}
				The above analysis can be easily adapted to claim the ISS robustness of \algo/ in the case of perfect networks.
				Indeed, the key element in the proof of Theorem~\ref{th:error} is the exploitation of the global exponential stability of $(0,\pzeq(0))$ for~\eqref{eq:ralgo_transformed} (cf. Section~\ref{sec:proof_ralgo}). %
				Then, since in Section~\ref{sec:proof_algo} we prove that $(0,\pzeq(0))$ is globally exponentially stable for~\eqref{eq:sp_system}, i.e., for the system that equivalently describes the evolution of \algo/, the same steps of the proof of Theorem~\ref{th:error} can be used to assess the ISS of \algo/ with respect to the generic errors $\errxit$ and $\errzit$ considered in~\eqref{eq:algo_with_errrors_aggregate}.\oprocend
			\end{remark}

			\section{Numerical Simulations}
			\label{sec:numerical_simulations}

			\GC{
				In this section, we perform some numerical simulations to test the effectiveness of \algo/ and \ralgo/ and to compare them with other distributed algorithms.
				First, in Section~\ref{sec:quadratic}, we consider a quadratic scenario and compare \algo/ with the well-known \GT/ algorithm~\cite{shi2015extra,nedic2017achieving,xi2017add}.
				Indeed, as we better explain in the next, the quadratic scenario allows for choosing the parameters optimizing the algorithms' convergence rates and, thus, allows for a fair speed comparison.

				Then, we numerically confirm the arguments expressed in Section~\ref{sec:from_avg_to_admm} to motivate the design of our ADMM-based schemes.
				Indeed, we compare \ralgo/ with an alternative counterpart using the same optimization-oriented scheme~\eqref{eq:gt-with-consensus-x} but a robust push-sum mechanism (see, e.g.,~\cite{bof_multiagent_2019,tian_achieving_2020}) in place of our ADMM consensus.
				In particular, in Section~\ref{sec:logistic_regression}, we consider a logistic regression scenario with asynchronous agents and packet losses in the communication.
				Finally, in Section~\ref{sec:logistic_regression_errors}, we perform the same comparison considering also generic errors (due to, e.g., inexact communications and computations) affecting the algorithms' updates (see Section~\ref{sec:robustness}).
			}

			The comparisons are done in terms of the convergence of the optimality error norm $\norm{\xt - \oneNn\xstar}$.
			All the simulations are performed by considering networks of $N = 10$ agents and an underlying randomly generated \er/ graph with connectivity parameter $0.1$.
			In all the simulations, we run our schemes by randomly selecting $\x_i^0 \in \R^{\n}$ and $\z_{ij}^0 \in \R^{2\n}$ for all $i \in \until{N}$ and $j \in \cN_i$.
			\GC{As for \GT/ and the push-sum-based scheme, we run it by setting the same $\x_i^0$ set for our schemes, while their consensus mechanisms are initialized as prescribed in~\cite{shi2015extra,nedic2017achieving,xi2017add} and~\cite{bof_multiagent_2019,tian_achieving_2020}, respectively.}

			\subsection{Quadratic Scenario}
			\label{sec:quadratic}
			
			In this section, we consider a quadratic setup described by
			\begin{equation*}%
				\min_{x \in \R^\n}\:\: 
				\sum_{i=1}^{N} \left(\tfrac{1}{2}\x\T Q_i\x  + r_i\T x\right),
			\end{equation*}
			where $Q_i \in \R^{\n \times \n}$ and $r_i \in \R^\n$.
			Moreover, it holds $Q_i = Q_i\T > 0$ for all $i \in \until{N}$ and, thus, the problem is strongly convex.
			We set $\n = 2$ and, for all $i \in \until{N}$, we randomly generate each matrix $Q_i$ so that all its eigenvalues belong to the interval $[1,5]$, while the components of each vector $r_i$ are randomly generated from the interval $[-10,20]$ with a uniform distribution.
			\GC{For both schemes considered, we select the parameters that maximize the rate of convergence. 
			Indeed, since the quadratic scenario results in algorithm updates with a linear form, this can be achieved through a preliminary line search aimed at finding the parameters that minimize the non-unitary largest eigenvalue of the linear autonomous systems describing the algorithms' evolution in error coordinates.}
			Specifically, we choose $\pr = 0.865$, $\step = 0.865$, $\rho =  0.3029$, $\alpha = 0.865$ for \algo/, while we set $\pr =  0.106$ for \GT/.
			Fig.~\ref{fig:error_quadratic} reports the simulation results and shows that \algo/ outperforms \GT/ in terms of convergence rate.
			\begin{figure}[htpb]
				\centering
				\includegraphics[scale=1]{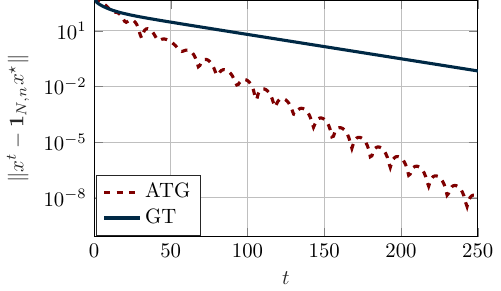}
				\caption{
					Quadratic scenario: comparison between \algo/ (\ATG/) and \GT/ (GT).
				}
				\label{fig:error_quadratic}
			\end{figure}
            
			\subsection{Logistic Regression Scenario with Asynchronous and Lossy Networks}
\label{sec:logistic_regression}

			In this section, we consider a logistic regression scenario in the case of imperfect networks with asynchronous agents and packet losses (see Section~\ref{sec:asynchronous_and_lossy}).
			A network of agents aims to cooperatively train a linear classifier for a set of points in a given feature space. 
			Each agent $i$ is equipped 
			with $m_i \in \N$ points $p_{i,1}, \dots, p_{i,m_i} \in \R^{\n-1}$ with binary 
			labels $l_{i,k} \in \{-1,1\}$ for all $k \in \until{m_i}$. 
			The problem consists of building a linear classification model from these points solving the minimization problem described by
			\begin{equation*}%
				\min_{w \in \R^{\n-1},b \in \R}\:\: 
				\sum_{i=1}^{N}\sum_{k=1}^{m_i}\log\left(1 + e^{-l_{i,k}(w^\top p_{i,k} + b)}\right) 
				+ \frac{C}{2}\norm{\begin{bmatrix}w\\b\end{bmatrix}}^2,
			\end{equation*}
			where $C > 0$ is the so-called regularization parameter. 
			Notice that the presence of the regularization term makes the cost function strongly convex. 
			We set $\n = 2$ and we randomly generate both the points and labels.
			For all $\iter\in\N$, $i \in \until{N}$, and $j \in \cN_i$ we generate each $\lit$ and $\bijt$ by extracting them from the set $\{0,1\}$ according to the binomial distribution with randomly generated probabilities of success which, in turn, are uniformly randomly generated from the interval $[0.1,1)$. 
			\GC{In this setup, we compare \ralgo/ with an alternative distributed algorithm combining the same optimization-oriented scheme~\eqref{eq:gt-with-consensus-x} with the robust push-sum consensus detailed in, e.g.,~\cite{bof_multiagent_2019,tian_achieving_2020}.
			We empirically tune the algorithms' parameters by choosing $\alpha = 0.9$, and $\rho = 0.9$ for \ralgo/, while we set $\pr = 0.1$ and $\step = 1$ for both \ralgo/ and the push-sum-based counterpart.}
			\GC{As predicted by Theorem~\ref{th:convergence_ralgo}, Fig.~\ref{fig:imperfect_netwroks} shows that \ralgo/ exhibits exact and linear convergence in the case of imperfect networks.
			Moreover, it also shows that \ralgo/ is faster than its push-sum-based counterpart.}
			\begin{figure}[H]
			\centering
			\includegraphics[scale=1]{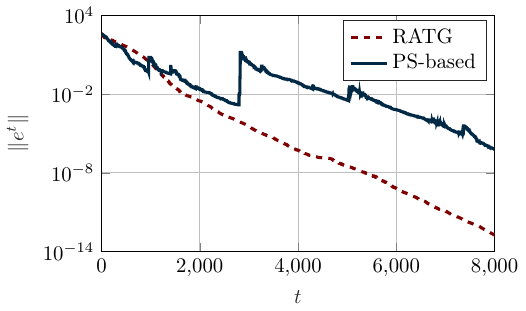}
			\caption{\GC{Logistic regression scenario with imperfect networks: comparison between \ralgo/ (\RATG/) and the push-sum-based counterpart (PS-based).}}
			\label{fig:imperfect_netwroks}
			\end{figure}

			\subsection{Logistic Regression Scenario with Inexactness and Asynchronous and Lossy Networks}
			\label{sec:logistic_regression_errors}

			In this section, we test the robustness of \ralgo/ by considering the same logistic regression problem and the same imperfect network of Section~\ref{sec:logistic_regression} and also the presence of generic disturbances affecting the algorithm updates as generically detailed in Section~\ref{sec:robustness}.
\GC{In detail, for each component $k$ of each algorithmic state of both algorithms, we consider a disturbance $e_k^t$ extracted according to the Gaussian distribution with $0$ mean and $\epsilon$ variance.}
\GC{We empirically tune the algorithms' parameters by choosing $\alpha = 0.9$, and $\rho = 0.9$ for \ralgo/, while we set $\pr = 0.05$ and $\step = 1$ for both \ralgo/ and the push-sum-based counterpart.}
\GC{Fig.~\ref{fig:inexact_1em4} and~\ref{fig:inexact_1em2} compares the algorithms' performance in the case of disturbances characterized by $\epsilon = 0.0001$ and $\epsilon = 0.01$, respectively.
In both cases, as predicted by Theorem~\ref{th:error}, we note that \ralgo/, differently from its push-sum-based counterpart, exhibits the typical behavior of a system with ISS properties.}
\begin{figure}[htpb]
	\centering
	\includegraphics[scale=1]{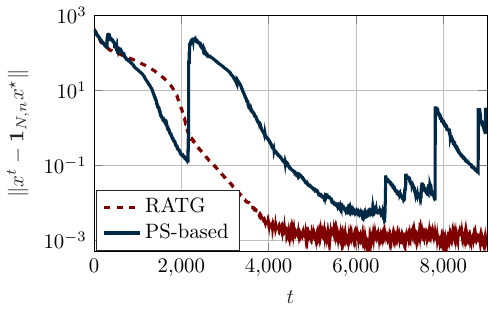}
	\caption{
		\GC{Logistic regression scenario with imperfect networks and inexactnesses: comparison between \ralgo/ (\RATG/) and the push-sum-based counterpart (PS-based) in the case of Gaussian disturbances with variance $\epsilon=0.0001$.}
	}
	\label{fig:inexact_1em4}
\end{figure}
\begin{figure}[htpb]
	\centering
	\includegraphics[scale=1]{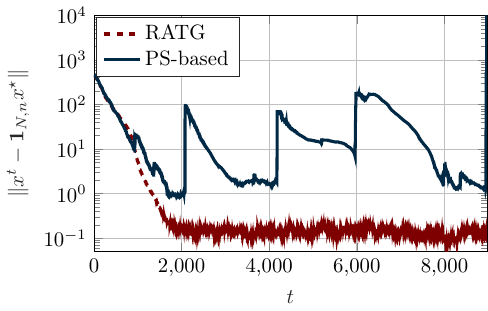}
	\caption{
		\GC{Logistic regression scenario with imperfect networks and inexactnesses: comparison between \ralgo/ (\RATG/) and the push-sum-based counterpart (PS-based) in the case of Gaussian disturbances with variance $\epsilon=0.01$.}
	}
	\label{fig:inexact_1em2}
\end{figure}

			\section{Conclusions}
			
			In this paper, we proposed a novel distributed algorithm for consensus optimization and a robust extension to deal with imperfect networks. 
			In detail, we designed the algorithm by interpreting the dynamic consensus problem as an auxiliary optimization program addressed through a distributed version of the ADMM algorithm.
			We interlaced the obtained scheme with suitable local, proportional actions giving rise to the novel distributed algorithm that we named \algo/.
			In the case of strongly convex problems, we proved the linear convergence of the algorithm through a system theoretical analysis based on timescale separation.
			Then, we considered the case in which the agents' updates are asynchronous and their communication is affected by packet losses.
      		For this scenario, we proposed \ralgo/, namely, a robust version of the first scheme and proved that such a method preserves linear convergence.
			Moreover, by relying on ISS concepts, we demonstrated that our schemes are robust with respect to generic additive errors.
			Finally, numerical tests confirmed our theoretical findings and compared both \algo/ and \ralgo/ with other state-of-the-art distributed algorithms.

			\appendix

			\subsection{Distributed R-ADMM: Algorithm Description}
			\label{sec:admm}

			This section briefly recalls the so-called Distributed R-ADMM method proposed in~\cite{bastianello2020asynchronous}.
			Let us consider a network of $N$ agents communicating according to the graph $\cG := \{\until{N},\cE\}$, with $\cE \subset \until{N} \times \until{N}$, that aim to solve in a distributed fashion the optimization problem
			\begin{align}\label{eq:problem_generic_admm}
				\begin{split}
					\min_{
							(\x_1,\dots,\x_N) \in \R^{N\n}
					}
					&\sum_{i=1}^N \ell_i(\x_i)
					\\
					\text{s.t.:} \: \: &\x_i = \x_j  \: \forall (i,j) \in \cE,
				\end{split}
			\end{align}
			where $\ell_i: \R^{\n} \to \R$ for all $i \in \until{N}$.
			When running Distributed R-ADMM, each agent $i$, at each iteration $\iter \in \N$, maintains a local estimate $\xit \in \R^{\n}$ about a solution to problem~\eqref{eq:problem_generic_admm} and an auxiliary variable $\zijt$ for each $j \in \cN_i$.
			These quantities are updated according to
			\begin{subequations}\label{eq:ADMM_generic}
				\begin{align}
					\xitp &= \argmin_{\x_i \in \R^\n}\left\{\ell_i(\x_i) \! - \!  \x_i\T \sum_{j \in \cN_i}\zijt \! + \! \frac{\rho \dii}{2} \norm{\x_i}^2\right\}
					\\
					\zijtp &= (1 - \alpha)\zijt - \alpha\zjit + 2\alpha\rho\xjtp,
				\end{align}
			\end{subequations}
			where $\alpha \in (0,1)$ and $\rho > 0$ are tuning parameters.

			\subsection{Proof of Theorem~\ref{th:theorem_generic}}
			\label{sec:proof_generic}

			Let us define $\tzt \coloneqq \zt - \zeq(\xt)$ and rewrite~\eqref{eq:interconnected_system_generic} as 
      \begin{subequations}\label{eq:interconnected_system_generic_err}
        \begin{align}
          \xtp &= \xt + \pr f(\xt,\tzt + \zeq(\xt),\iter)\label{eq:slow_system_generic_err}
          \\
          \tztp &= g(\tzt  +  \zeq(\xt),\xt,\iter) -  \zeq(\xt)  +  \Delta \zeq(\xtp,\xt),\label{eq:fast_system_generic_err}
        \end{align}
      \end{subequations}
      where $\Delta \zeq(\xtp,\xt) \coloneqq - \zeq(\xtp) + \zeq(\xt)$. Pick $W$ as in~\eqref{eq:W_generic}. 
			By evaluating $\Delta W(\xt,\iter) \coloneqq W(\xtp,\iterp) - W(\xt,\iter)$ along the trajectories of~\eqref{eq:slow_system_generic_err}, we obtain 
      \begin{align}
        &
        \Delta W(\xt,\iter) 
        \notag\\
        &= W(\xt + \pr f(\xt,\tzt + \zeq(\xt),\iter),\iterp) 
				- W(\xt,\iter)
        \notag\\
        &\stackrel{(a)}{=} 
        W(\xt + \pr f(\xt,\zeq(\xt),\iter),\iterp) - W(\xt,\iter) 
        \notag\\
        &\hspace{.4cm}
        + W(\xt + \pr f(\xt,\tzt + \zeq(\xt),\iter),\iterp) 
        \notag\\
        &\hspace{.4cm}
        - W(\xt + \pr f(\xt,\zeq(\xt),\iter),\iterp)
        \notag\\
        &\stackrel{(b)}{\leq}
        -\pr c_3\norm{\xt}^2 + W(\xt + \pr f(\xt,\tzt + \zeq(\xt),\iter),\iterp) 
         \notag\\
        &\hspace{.4cm}
        - W(\xt + \pr f(\xt,\zeq(\xt),\iter),\iterp)
				\notag\\
        &\stackrel{(c)}{\leq}
        -\pr c_3\norm{\xt}^2 + 2\pr c_4 \lip_f\norm{\tzt}\norm{\xt} 
         \notag\\
         &\hspace{.4cm}
        + \pr^2c_4 \lip_f\norm{\tzt}\norm{ f(\xt,\tzt + \zeq(\xt),\iter)}
        \notag\\
        &\hspace{.4cm}
        + \pr^2c_4\lip_f\norm{\tzt} \norm{ f(\xt, \zeq(\xt),\iter)},\label{eq:deltaU_generic}
      \end{align}
      where in $(a)$ we add and subtract the term $W(\xt + \pr f(\xt,\zeq(\xt),\iter),\iterp)$, in $(b)$ we exploit~\eqref{eq:W_minus_generic} to bound the difference of the first two terms, in $(c)$ we use~\eqref{eq:W_bound_generic}, the Lipschitz continuity of $f$, and the triangle inequality. 
			By recalling that $f(0, \zeq(0),\iter) = 0$ we can write 
      \begin{align}
        &\norm{f(\xt, \tzt + \zeq(\xt),\iter)}
				\notag\\
				 &= \norm{ f(\xt, \tzt + \zeq(\xt),\iter) - f(0, \zeq(0),\iter)}
        \notag\\
        &\stackrel{(a)}{\leq}
        \lip_f\norm{\xt} + \lip_f\norm{\tzt + \zeq(\xt) - \zeq(0)},\nonumber \\
        &\stackrel{(b)}{\leq}
        \lip_f(1 + \lip_{\text{eq}})\norm{\xt} + \lip_f\norm{\tzt},\label{eq:bound_Lf_L_h_wt}
      \end{align}
      where in $(a)$ we use the Lipschitz continuity of $f$ and $\zeq$, and in $(b)$ we use the Lipschitz continuity of $\zeq$ together with the triangle inequality. 
			With similar arguments, we have
      \begin{align}
        \norm{ f(\xt, \zeq(\xt),\iter)} \leq \lip_f (1+ \lip_{\text{eq}})\norm{\xt}.\label{eq:bound_L_f_L_h}
      \end{align}
      Using inequalities~\eqref{eq:bound_Lf_L_h_wt} and~\eqref{eq:bound_L_f_L_h} we then bound~\eqref{eq:deltaU_generic} as 
      \begin{align}
        &
				\Delta W (\xt,\iter) 
				\notag\\
				&\leq -\pr c_3\norm{\xt}^2 + 2\pr c_4 \lip_f\norm{\tzt}\norm{\xt} 
        + \pr^2c_4 \lip_f^2\norm{\tzt}^2
        \notag\\
        &\hspace{.4cm}
        +  2 \pr^2c_4 \lip_f^2(1+\lip_{\text{eq}})\norm{\tzt}\norm{\xt}
        \notag\\
        &\leq
        -c_3\norm{\xt}^2 \! + \! \pr^2 k_3\norm{\tzt}^2
        \! + \! (\pr k_1 + \pr^2 k_2)\norm{\tzt}\norm{\xt}, \label{eq:deltaU_generic_final} 
      \end{align}
      where we introduce the constants
      \begin{align*}
        k_1 &\coloneqq 2c_4\lip_f,
        \quad
        k_2 \coloneqq 2c_4 \lip_f^2(1+\lip_{\text{eq}}),
        \quad%
        k_3 \coloneqq c_4 \lip_f^2.
      \end{align*}
      We now evaluate $\Delta U(\tzt,\iter) \coloneqq U(\tztp,\iterp) - U(\tzt,\iter)$ (see~\eqref{eq:U_generic}) along the trajectories of~\eqref{eq:fast_system_generic_err}, obtaining
      \begin{align}
        &\Delta U(\tzt,\iter) 
        \notag\\
        &
        = U(g(\tzt + \zeq(\xt),\xt,\iter) -\zeq(\xt) + \Delta \zeq(\xtp,\xt),\iterp) 
        \notag\\
        &\hspace{.4cm}
        - U(\tzt,\iter)
        \notag\\
        &\stackrel{(a)}{\leq}  U(g(\tzt + \zeq(\xt),\xt),\iter - \zeq(\xt),\iterp) - U(\tzt,\iter) 
        \notag\\
        &\hspace{.4cm}
        - U(g(\tzt + \zeq(\xt),\xt,\iter) - \zeq(\xt),\iterp)
        \notag\\
        &\hspace{.4cm}
        + U(g(\tzt + \zeq(\xt),\xt,\iter) -\zeq(\xt) + \Delta \zeq(\xtp,\xt),\iterp)
        \notag\\        
        &\stackrel{(b)}{\leq}  -b_3\norm{\tzt}^2  - U(g(\tzt + \zeq(\xt),\xt,\iter) - \zeq(\xt),\iterp)
        \notag\\
        &\hspace{.4cm}
        + U(g(\tzt + \zeq(\xt),\xt,\iter) -\zeq(\xt) + \Delta \zeq(\xtp,\xt),\iterp) 
        \notag\\
        &\stackrel{(c)}{\leq}  -b_3\norm{\tzt}^2  
        \notag\\
        &\hspace{.4cm}
        + b_4\norm{\Delta \zeq(\xtp,\xt)}
         \notag\\
        &\hspace{.8cm}
        \times 
        \norm{g(\tzt + \zeq(\xt),\xt,\iter) - \zeq(\xt)+ \Delta \zeq(\xtp,\xt)} 
        \notag\\
        &\hspace{.4cm}
        + b_4\norm{\Delta \zeq(\xtp,\xt)}\norm{g(\tzt + \zeq(\xt),\xt,\iter) - \zeq(\xt)}
        \notag\\
        &\stackrel{(d)}{\leq}  -b_3\norm{\tzt}^2  + b_4\norm{\Delta \zeq(\xtp,\xt)}^2 
        \notag\\
        &\hspace{.4cm}
        + 2 b_4\norm{\Delta \zeq(\xtp,\xt)}\norm{g(\tzt  +  \zeq(\xt),\xt,\iter)  -  \zeq(\xt)},\label{eq:DeltaW_generic}
      \end{align}
      where in $(a)$ we add and subtract $U(g(\tzt + \zeq(\xt),\xt,\iter)-\zeq(\xt),\iterp)$, in $(b)$ we exploit~\eqref{eq:U_minus_generic} to bound the first two terms, in $(c)$ we use~\eqref{eq:U_bound_generic} to bound the difference of the last two terms, and in $(d)$ we use the triangle inequality. 
			The definition of $\Delta \zeq(\xtp,\xt)$ and the Lipschitz continuity of $\zeq$ lead to
      \begin{align}
        &\norm{\Delta \zeq(\xtp,\xt)} 
				\notag\\
				&\leq \lip_{\text{eq}}\norm{\xtp - \xt} 
        \notag\\
        &\stackrel{(a)}{\leq}
        \pr \lip_{\text{eq}}\norm{f(\xt,\tzt+\zeq(\xt),\iter)}
        \notag\\
        &\stackrel{(b)}{\leq}
        \pr \lip_{\text{eq}}\norm{f(\xt,\tzt+\zeq(\xt),\iter) - f(0, \zeq(0),\iter)}
        \notag\\
        &\stackrel{(c)}{\leq}
        \pr \lip_{\text{eq}}\lip_f(1  +  \lip_{\text{eq}})\norm{\xt}  +  \pr \lip_{\text{eq}}\lip_f\norm{\tzt},\label{eq:bound_Deltah}
      \end{align}
      where in $(a)$ we use the update~\eqref{eq:slow_system_generic_err}, in $(b)$ we add the term $f(0,\zeq(0),\iter)$ since this is zero, and in $(c)$ we use the triangle inequality and the Lipschitz continuity of $f$ and $\zeq$. 
			Moreover, since $g(\zeq(\xt),\xt,\iter) = \zeq(\xt)$, we obtain
      \begin{align}
        &\norm{g(\tzt + \zeq(\xt),\xt,\iter) - \zeq(\xt)} 
        \notag\\
        &
        =\norm{g(\tzt + \zeq(\xt),\xt,\iter) - g(\zeq(\xt),\xt,\iter)}
				\notag\\
        &
        \leq 
        \lip_g\norm{\tzt},\label{eq:bound_g}
      \end{align}
      where the inequality is due to the Lipschitz continuity of $g$. 
			Using inequalities~\eqref{eq:bound_Deltah} and~\eqref{eq:bound_g}, we then bound~\eqref{eq:DeltaW_generic} as 
      \begin{align}
        \Delta U(\tzt,\iter) 
        &\leq -b_3\norm{\tzt}^2  + 2\pr b_4 \lip_{\text{eq}} \lip_g\lip_f(1 + \lip_{\text{eq}})\norm{\xt}\norm{\tzt} 
        \notag\\
        &\hspace{.4cm}
        + 2\pr b_4 \lip_{\text{eq}} \lip_g\lip_f\norm{\tzt}^2
				\notag\\
        &\hspace{.4cm}
        +\pr^2b_4 \lip_{\text{eq}}^2 \lip_f^2(1 + \lip_{\text{eq}})^2\norm{\xt}^2 
        \notag\\
        &\hspace{.4cm}
        + 2\pr^2b_4 \lip_{\text{eq}}^2 \lip_f^2(1 + \lip_{\text{eq}})\norm{\xt}\norm{\tzt} 
				\notag\\
        &\hspace{.4cm}
				+ \pr^2 b_4 \lip_{\text{eq}}^2 \lip_f^2\norm{\tzt}^2
        \notag
				\\
        &\leq
        (-b_3 +\pr k_6 + \pr^2 k_7)\norm{\tzt}^2 + \pr^2 k_8\norm{\xt}^2
        \notag\\
        &\hspace{.4cm}
        + (\pr k_4 + \pr^2 k_5)\norm{\xt}\norm{\tzt}, \label{eq:deltaW_generic_final}
      \end{align}
      where we introduce the constants 
      \begin{align*}
        k_4 &\coloneqq 2b_4 \lip_{\text{eq}}\lip_g\lip_f(1 + \lip_{\text{eq}}), \quad
        &&k_5 \coloneqq 2b_4 \lip_{\text{eq}}^2\lip_f^2(1 + \lip_{\text{eq}}),
        \\
        k_6 &\coloneqq 2b_4 \lip_{\text{eq}}\lip_g\lip_f,
        \quad
        &&k_7 \coloneqq b_4 \lip_{\text{eq}}^2\lip_f^2,
        \\
        k_8 &\coloneqq b_4 \lip_{\text{eq}}^2\lip_f^2(1+\lip_{\text{eq}})^2.
      \end{align*}
      We now consider $V: \mathcal{D} \times \R^m \times \N \to \R$ defined as
      \begin{align*}
        V(\xt,\tzt,\iter) = W(\xt,\iter) + U(\tzt,\iter). 
      \end{align*}
      By evaluating $\Delta V(\xt,\tzt,\iter) \coloneqq V(\xtp,\tztp,\iterp) - V(\xt,\tzt,\iter) = \Delta W(\xt,\iter) + \Delta U(\tzt,\iter)$ along the trajectories of~\eqref{eq:interconnected_system_generic_err}, we can use the results~\eqref{eq:deltaU_generic_final} and~\eqref{eq:deltaW_generic_final} to write 
      \begin{align}
        \Delta V(\xt,\tzt,\iter) &\leq -\begin{bmatrix}
          \norm{\xt}\\
          \norm{\tzt}
        \end{bmatrix}\T Q(\pr) \begin{bmatrix}
          \norm{\xt}\\
          \norm{\tzt}
        \end{bmatrix},\label{eq:deltaV_generic}
      \end{align}
      where we define the matrix $\cQ(\pr) = \cQ(\pr)\T \in \R^2$ as 
      \begin{align*}
        &\cQ(\pr) \coloneqq \begin{bmatrix}
          \pr c_3 -\pr^2k_8& q_{21}(\pr)\\
          q_{21}(\pr) & b_3 - \pr k_6 - \pr^2 (k_3 + k_7)
        \end{bmatrix},
      \end{align*}
      with $q_{21}(\pr) \coloneqq -\frac{1}{2}(\pr (k_1+k_4) + \pr^2 (k_2+k_5))$. By Sylvester criterion~\cite{khalil2002nonlinear}, we know that $\cQ > 0$ if and only if 
      \begin{align}\label{eq:conditions_generic}
        \pr c_3b_3 > p(\pr)
      \end{align}
      where the polynomial $p(\pr)$ is defined as 
      \begin{align}
        p(\pr) &\coloneqq q_{21}(\pr)^2 + \pr^2 c_3k_6 
        + \pr^3 c_3(k_3 + k_7)+ \pr^2b_3k_8
                \notag\\
        &\hspace{.4cm} 
        -\pr^3k_6k_8 - \pr^4k_8(k_3+k_7).
      \end{align}
      We note that $p$ is a continuous function of $\pr$ and $\lim_{\pr \to 0}p(\pr)/\pr = 0$. 
	  Hence, there exists some $\bar{\pr} \in (0,\bar{\pr}_1)$ so that~\eqref{eq:conditions_generic} is satisfied for all $\pr \in (0,\bar{\pr})$. Under such a choice of $\pr$, and denoting by $q_{\text{small}} > 0$ the smallest eigenvalue of $\cQ(\pr)$, we bound~\eqref{eq:deltaV_generic} as 
      \begin{align*}
        \Delta V(\xt,\tzt,\iter) \leq -q_{\text{small}}\norm{\begin{bmatrix}\norm{\xt}\\
            \norm{\tzt}\end{bmatrix}}^2,
      \end{align*}
      which allows us to conclude, in view of \cite[Theorem~13.2]{chellaboina2008nonlinear}, that $(0, 0)$ is an exponentially stable equilibrium point for system~\eqref{eq:interconnected_system_generic_err}. 
	  The theorem's conclusion follows then by considering the definition of exponentially stable equilibrium point and by reverting to the original coordinates $(\xt,\zt)$.

\end{document}